\documentclass[12pt]{amsart} \usepackage{amscd}
\usepackage[all]{xy}
\usepackage{amssymb}
\UseComputerModernTips

\newtheorem{theorem}{Theorem}
\newtheorem{lemma}[theorem]{Lemma}
\newtheorem{corollary}{Corollary}
\newtheorem{proposition}[theorem]{Proposition}

\theoremstyle{definition}                 
            
\newtheorem{remark}{Remark} \newtheorem*{notation}{Notation}

\usepackage{amsfonts}
\newcommand{\field}[1]{\mathbb{#1}}          \newcommand{\Q}{\field{Q}}
\newcommand{\R}{\field{R}}                   \newcommand{\Z}{\field{Z}}
\newcommand{\C}{\field{C}}

                    \newcommand{\e}{\varepsilon}

 \newcommand{\ra}{\rightarrow}

\begin{document}

\title[Monodromy of Cyclic Coverings of the Projective Line]
{Monodromy of Cyclic Coverings of the Projective Line}

\author{T.N.Venkataramana}

{\address{ T.N.Venkataramana, School of Mathematics, TIFR, Homi Bhabha
Road, Colaba, Mumbai 400005, India}

\email{venky@math.tifr.res.in}

\subjclass{primary:  22E40.  Secondary:  20F36  \\  T.N.Venkataramana,
School of  Mathematics, Tata  Institute of Fundamental  Research, Homi
Bhabha Road, Colaba, Mumbai 40005, INDIA}

\date{}

\begin{abstract} We show that the  image of the pure braid group under
the monodromy  action on the homology  of a cyclic  covering of degree
$d$ of the projective line  is an arithmetic group provided the number
of ramification  points is sufficiently  large compared to  the degree
$d$ and the ramification degrees are co-prime to $d$.
\end{abstract}

\maketitle{}

\section{Introduction}\label{introsection}

A subgroup $\Gamma \subset GL_N(\Z)$, is said to be an {\bf arithmetic
group} if $\Gamma  $ has finite index in  its integral Zariski closure
${\mathcal G}(\Z)$  (i.e.  suppose  ${\mathcal G}\subset GL_N$  is the
Zariski closure of $\Gamma  $; then $\Gamma \subset {\mathcal G}(\Z)$,
which  by definition, is  ${\mathcal G}\cap  GL_N(\Z)$.  We  say that
$\Gamma $ is arithmetic, if $\Gamma $ is a subgroup of finite index in
${\mathcal  G}(\Z)$).   Otherwise,  we  say  that  $\Gamma  $  is  not
arithmetic, or that $\Gamma $ is {\it thin} \cite{Sar}. \\

A natural class of subgroups  of $SL_N(\Z)$ arise as monodromy groups.
Suppose  $X\ra   S$  is  a  family  of   smooth  projective  varieties
$(X_s)_{s\in  S}$  parametrised  by  a  base variety  $S$.   Then  the
fundamental  group  $\pi  _1(S)$   acts  on  the  integral  cohomology
$H^*(X_s)$ of a  typical fibre $X_s$. The image of  this action is the
``monodromy group''.  Griffiths  and Schmid \cite{Gr-Sch} first raised
the possibility that monodromy  groups are arithmetic.  However, there
are several examples which show that the monodromy group is not always
an arithmetic  group.  Notable among them are  those of Deligne-Mostow
\cite{Del-Mos} (see  also \cite{Nor} where the monodromy  group is not
even  finitely presented).   In  the examples  of \cite{Del-Mos},  the
monodromy  group is  a subgroup  of  infinite index  in an  arithmetic
lattice in a  product of unitary groups $U(r,s)$  (such that the group
of  the real  points of  the Zariski  closure of  the  monodromy group
contains the  product of  the special unitary  groups $SU(p,q)$  and )
such  that one  of  the factors  of  the product  is $U(n-1,1)$.   The
projection of the  monodromy to this factor sometimes  gives a lattice
in $U(n-1,1)$  which can  be shown to  be a non-arithmetic  lattice in
$U(n-1,1)$. \\

The  examples of  \cite{Del-Mos}  arise as  the  monodromy of  certain
families of  cyclic coverings of a  fixed order $d$  of the projective
line ${\mathbb  P}^1(\C)$, where the family is  prescribed by choosing
$n+1$  distinct branch  points in  the  affine line  $\C$, with  fixed
ramification. To be precise, let  $n\geq 1$ and $d\geq 2$ be integers.
Fix integers $k_1,k_2,\cdots, k_{n+1}$  with $1\leq k_i \leq d-1$, and
such that  the g.c.d.  of $k_1,k_2,  \cdots, k_{n+1}$ and  $d$ is $1$.
Given $n+1$ distinct points  $a_1,a_2, \cdots, a_{n+1}$ in the complex
plane, put $a=(a_1, \cdots, a_{n+1}) \in \C^{n+1}$. Consider the curve
$X_{a,k}$ given by the pair ($x,y$) satisfying the equation
\begin{equation} \label{affine} 
y^d= (x-a_1)^{k_1}(x-a_2)^{k_2}\cdots (x-a_{n+1})^{k_{n+1}}\end{equation}
with $y\neq 0$ and $x\neq a_1,a_2, \cdots, a_{n+1}$. 

Let ${\mathcal C}$  be the space of points in  $\C^{n+1}$ all of whose
coordinates are distinct; as the  point $a\in {\mathcal C}$ varies, we
get a family
\[{\mathcal F}=\{(y,x,a)\in \C ^*
\times   \C    \times   {\mathcal   C}:    y^d=   \prod   _{i=1}^{n+1}
(x-a_i)^{k_i}\},\]  and  the  fibration  ${\mathcal  F}\ra
{\mathcal  C}$ given  by the  projection map  $(y,x,a)\mapsto  a$. The
fibre over  a point $a\in {\mathcal  C}$ is the affine  curve given in
equation  (\ref{affine}).  The  curve $X_{a,k}$  is a  compact Riemann
surface $X_{a,k}^*$ minus a finite set of punctures.  We may consider,
analogously, the  family ${\mathcal F}^*$ of  compact Riemann surfaces
$X_{a,k}^*$ fibering over ${\mathcal C}$. \\

The fundamental group of the space  ${\mathcal C}$ is well known to be
the  pure  braid group  $P_{n+1}$  on  $n+1$  strands (see  subsection
\ref{topologybraid}); thus the fibration ${\mathcal F}^* \ra {\mathcal
C}$ yields a monodromy representation
\[\rho _M ^* (k,d): P_{n+1}\ra GL(H_1(X_{a,k} ^* ,\Z)),\] of $P_{n+1}$
on the integral homology of the fibre $X_{a,k}^*$.  If $N$ is the rank
of the abelian group $H_1(X_{a,k}^*,\Z)$,  then the image $\Gamma $ of
$P_{n+1}$ is a subgroup of $GL_N(\Z)$.  It can be shown that the group
${\mathcal G}(\R)$ of  real points of the Zariski  closure of $\Gamma$
is  contained  in a  product  of  unitary  groups $U(p,q)$  such  that
${\mathcal G}(\R)$ contains the  product of the special unitary groups
$SU(p,q)$. The group $G=\Z/d\Z$ acts on the equation
\[y^d= \prod _{i=1}^ {n+1}(x-a_i)^{k_i}  ,\] by the map $g(x,y)\mapsto
gy$ for $g\in G$ where $G$ is viewed as $d$-th roots of unity in $\C$.
We  may   decompose  the  homology  $H_1$  of   $X_{a,k}^*$  into  $G$
eigenspaces with respect to this  action.  Fix a primitive $d$-th root
of unity, say $\omega =e^{2\pi i  /d}$. Fix a generator $T$ of $G$. If
$1\leq  f  \leq  d$ is  an  integer,  fix  the  part of  the  homology
$H_1(X_{a,k}^*,\C)$ on which the generator $T\in G$ acts by the scalar
$\omega  ^f$. The  group  of real  points  of the  Zariski closure  of
$\Gamma $  acting on this  part will again  be contained in  a unitary
group  of the  form $U(p_f,q_f)$  and  will (in  general) contain  the
special unitary group $SU(p_f,q_f)$. \\

We now  describe briefly, the results of  \cite{Del-Mos}.  Suppose $f$
is an integer with $1\leq f \leq d-1$, and coprime to $d$. Given $x\in
{\mathbb  R}$,  denote  by  $\{x\}$  its  fractional  part.  Put  $\mu
_i=\{\frac{k_if}{d}\}$   for   $1\leq  i   \leq   n+1$.   Write   $\mu
_{\infty}=2- \sum \mu  _i$. We impose the following  conditions on the
$\mu _i$.  [1]  $\mu _i+\mu _j <1$ for  all $i,j$ including $i=\infty$
[2] $0< \mu _{\infty}$ [3]  $\frac{1}{1-\mu _i -\mu _j}$ is an integer
if $k_i\neq k_j$ [4] if $k_i=k_j$ then $\frac{1}{1-2\mu _i}$ is a half
integer.  \\

Then it  is shown in  \cite{Del-Mos} that the  factor of the  group of
real  points of  the Zariski  closure of  the monodromy  $\Gamma  $ in
$GL(H_1(X_{a,k}^*,\Z))=GL_N(\Z)\subset GL_N(\C)$  corresponding to $f$
(as  in the preceding  paragraph) contains  the special  unitary group
$SU(n-1,1)$ and  is contained in $U(n-1,1)$.  Moreover, the projection
of $\Gamma $ to this factor gives a lattice in $U(n-1,1)$ (if the $\mu
_i$ satisfy some further conditions, then the lattice in $U(n-1,1)$ is
an {\it arithmetic} lattice). \\

For example,  consider the family for  varying $b_1,b_2,b_3,b_4\in \C$, 
all distinct, of the curves corresponding to the equation
\[y^{18}=(x-b_1)(x-b_2)(x-b_3)(x-b_4).\] In this  case, $n=3$.  By the
criteria of  \cite{Del-Mos}, the monodromy is  non-arithmetic.  In the
notation of  the preceding paragraph,  we take $f=7$ and  $d=18$; then
$\mu _i=7/18$  and $\mu _{\infty}=8/18$.   Hence $\frac{1}{1-\mu _i-\mu
_j}$ is a half integer -namely  $9/2$ - if $i,j \leq 4$ (and hence $\mu
_i=\mu _j$);  and $\frac{1}{1-\mu _i  -\mu _j}$ is an  integer- namely
$6$- if  $i\leq 4$ and  $j=\infty$.  By the half  integrality ($\Sigma
-INT$) criterion  of Mostow (in  \cite{Mos}, see page 104,  with $N=5$
and  $\mu _i=7/18$,  and  $\mu _{\infty}=8/18$)  it  follows that  the
projection  of $\Gamma $  to the  factor corresponding  to $f=7$  is a
discrete subgroup of $U(2,1)$ and is, in fact, a lattice in
$U(2,1)$. One  can easily check, from  the list given  there, that the
monodromy is  non-arithmetic i.e.  has infinite index  in its integral
Zariski closure. \\

Let us now return to the general situation of equation (\ref{affine}).
The  condition   of  \cite{Del-Mos}  that   $0<  \mu  _{\infty}=2-\sum
\{\frac{k_if}{d}\}$ implies that $n+1\leq  \sum k_i \leq 2d$ and hence
that $n\leq  2d-1$.  We  would like to  investigate what  happens when
$n\geq 2d$.   The following theorem says  that if $n\geq  2d$ then for
most $k_i$'s, the monodromy is arithmetic. Precisely, we prove
\begin{theorem} \label{cyclicmonodromy}  Suppose $d\geq 2$  and $n\geq
1$  are integers,  $k_1,  \cdots, k_{n+1}$  are  integers with  $1\leq
k_i\leq d-1$ with $g.c.d (k_i,d)=1$ for each $i$.
Suppose that \[ n\geq 2d.\]
Then  the  image  $\Gamma  =\rho _M^*(k,d)(P_{n+1})$   of  the  monodromy
representation $\rho _M^* (k,d)$ of $P_{n+1}$ is an arithmetic group. \\

Moreover, the  monodromy group  is (up to  finite index) a  product of
irreducible lattices each of which has $\Q$-rank at least two.
\end{theorem}

In  \cite{Ve2}, the  case  when all  the  integers $k_i$  are $1$  was
considered (then  $g.c.d (k_i,d)=1$  for  all  $i$).  Consider  the
compactification $X_a^*$ of the affine curve
\[y^d= (x-a_1)(x-a_2)\cdots (x-a_{n+1}), \]  with $y\neq 0$ and $x\neq
a_1, \cdots, a_{n+1}$.  There is now the monodromy  action of the pure
braid  group $P_{n+1}$  (even of  the full  braid group  $B_{n+1}$) on
$H_1(X_a^*,\Z)$. The following result is proved in \cite{Ve2}. 

\begin{theorem}  \label{fullbraidmonodromy} If  $d\geq  3$ and  $n\geq
2d$, then  the image $\Gamma  $ of the monodromy  representation $\rho
(d):B_{n+1}\ra GL(H_1(X_a^*,\Z))=GL_N(\Z)$ is an arithmetic group. \\

Moreover, the  monodromy is  a finite index  subgroup of a  product of
irreducible  lattices, each  of which  is a  non-co-compact arithmetic
group and has $\Q$-rank at least two.
\end{theorem}

\begin{remark} If $n+1\leq d$ then the group of integral points of the
Zariski closure of the monodromy is  (up to finite index) a product of
irreducible arithmetic  lattices, some of which  form {\it co-compact}
lattices of their real Zariski closures. \\

A    result    of     A'Campo    \cite{A'C}    says    that    Theorem
\ref{fullbraidmonodromy} holds when $d=2$ as well.\\

If we replace the pure braid  group by the mapping class group $\Gamma
_g$ of  the fundamental  group of a  compact Riemann surface  of genus
$g\geq 2$, and consider analogously,  the action of $\Gamma _g$ on the
family of  cyclic coverings of a  fixed degree of the  family of genus
$g$  Riemann surfaces,  then the  arithmeticity of  the image  of this
action (monodromy) is proved in  \cite{Looi2} (at the time the present
article  was  written,   the  author  was  not  aware   of  the  paper
\cite{Looi2}; the method of proof  is similar and uses the presence of
unipotent  elements  in the  monodromy  group.   But,  in the  present
article, more work is needed to generate unipotent elements- under the
assumption that $n\geq 2d$).
\end{remark}

A special case of Theorem \ref{cyclicmonodromy} is the following

\begin{corollary} \label{primemonodromy} Suppose $d$ is a prime, $k_1,
\cdots, k_{n+1}$  integers with $1\leq  k_i \leq d-1$ and  $n\geq 2d$.
Then the monodromy group $\Gamma$, namely the image of $P_{n+1}$ under
the representation $\rho (k,d):P_{n+1}\ra GL(H_1(X_{a,k}^*,\Z))$ is an
arithmetic group.
\end{corollary}

\begin{remark} If $d$ is not assumed to be prime, then the analogue of
Corollary \ref{primemonodromy}  is false in general, even  when $n$ is
large.  As  an example,  consider $d=2\times 18$  and let $n$  be {\it
arbitrary}.   Suppose  $a_1,a_2,\cdots,  a_n,  b_1,  b_2,b_3,b_4$  are
distinct complex numbers. Consider the two equations
\[C_d(a,b): ~~y^{2\times       18}=      (\prod      _{i=1}^      n
(x-a_i))^{18}(x-b_1)\cdots (x-b_4)~~{\rm and}\]
\[C_{18}(b):   ~~w^{18}=(x-b_1)\cdots  (x-b_4).\]   There  is   a  map
$C_{2\times 18}(a,b)\ra C_{18}(b)$ given by $(x,y)\mapsto (x,w)$ with
\[w= \frac{y^2}{(x-a_1)\cdots  (x-a_n)} ~~ .  \] The  monodromy of the
family  $C_{2\times 18}(a,b)$  (as  $a$  and $b$  vary)  on the  first
homology   of  the   curves  $C_{2\times   18}(a,b)$  maps   onto  the
corresponding monodromy  of the family  of the curves  $C_{18}(b)$ (as
$b$ varies).  The  latter is not arithmetic, by  the example discussed
earlier.  Therefore, the monodromy of the family $C_{2\times 18}(a,b)$
is also non-arithmetic.
\end{remark}

\subsection{Description of the Proof}

The   proof    is   very   similar    to   the   proof    of   Theorem
\ref{fullbraidmonodromy} given in \cite{Ve2}.  In \cite{Ve2} the proof
was  by   showing  that  the   monodromy  was  related  to   the  Burau
representation. The properties of the Burau representation (especially
those at roots of unity) were used in the course of the proof. \\

Analogously,  in the present  paper, Theorem  \ref{cyclicmonodromy} is
deduced   from   the   arithmeticity   of  the   images   of   certain
representations  (the reduced  Gassner  representation specialised  at
roots of  unity defined in subsection  \ref{gassnersubsection}) of the
pure  braid group  $P_{n+1}$.   We also  have  to establish,  somewhat
precisely,  the  exact  relationship   of  the  monodromy  in  Theorem
\ref{cyclicmonodromy}  with the Gassner  representation. This  is much
more complicated than in the Burau case.  The monodromy representation
of  Theorem \ref{cyclicmonodromy}  is related  to the  reduced Gassner
representation  of  Theorem   \ref{purebraidmainth}  as  follows  (see
\cite{KM} for related results). \\

One can define the reduced Gassner representation $g_n(k,d)$ at $d$-th
roots  of   unity  where  $k$   is  the  $n+1$-tuple   $(k_1,  \cdots,
k_{n+1})$.  The  image  of  $g_n(k,d)$  takes  the  pure  braid  group
$P_{n+1}$  into $GL_n(E_d)$ where  $E_d=\Q (\omega  _d)$ is  the $d$-th
cyclotomic  extension.  We  will see  in  section \ref{gassnersection}
that  the Gassner  representation $g_n(k,d)$  is irreducible  if $\sum
k_i$ is {\it  not} divisible by $d$; if $\sum _i  k_i$ is divisible by
$d$,   then   $g_n(k,d)$   contains   the  one   dimensional   trivial
representation   $E_dv$   and   the  quotient,   denoted   ${\overline
g_n(k,d)}$,  is  irreducible.  By  an  abuse  of  notation,  we  write
${\overline  g_n(k,d)}$ for  the representation  $g_n(k,d)$  even when
$\sum k_i$ is not divisible by $d$.  \\

If $X_{a,k}$  is the  open curve, then  we have the  monodromy action,
denoted  $\rho _M(k,d)$  on  $H_1(X_{a,k},\Q)$ (and  the action  $\rho
_M^*(X_{a,k}^*,  \Q)$   on  the  homology  of   the  projective  curve
$X_{a,k}^*$).  On the homology  of $X_{a,k}$ the cyclic group $\Z/d\Z$
operates. Given a module $V$  of the $\Q$-group algebra $\Q [\Z/d\Z]$,
denote by $V^{ni}$  the quotient of $V$ modulo  the space of invariant
vectors    in   $V$    under   the    action   of    $\Z/d\Z$.    Take
$V=H_1(X_{a,k},\Q)$.  We call $V^{ni}$  the ``non-invariant''  part of
$V$ and denote  by $\rho _M(k,d)^{ni}$ the representation  of $P_{n+1}$ on
$V^{ni}$.  In section \ref{monodromysection}, we will prove

\begin{proposition}   \label{monodromyandgassner}  Suppose   that  the
numbers $k_i$  are all co-prime to  $d$. Denote by  $\rho _M(k,d)$ the
representation of  the pure braid  group $P_{n+1}$ on the  homology of
the  open  curve $H_1(X_{a,k},\Q)$.  Then  the  non-invariant part  of
$H_1(X_{a,k},\Q)$ is the direct sum
\[\rho   _M(k,d)   ^{ni}=\bigoplus  _{e   \mid   d}  g_n(k,e).\]   The
representation  $\rho  _M  ^*(k,d)$   of  $P_{n+1}$  on  the  homology
$H_1(X_{a,k}^*, \Q)$ of the compact Riemann surface $X_{a,k}^*$ is the
direct sum
\[\rho _M^*(k,d)=\bigoplus _{e \mid d} {\overline g_n(k,e)}.\]
The sum is over all divisors $e\geq 2$ of the integer $d$. 
\end{proposition}

Theorem     \ref{cyclicmonodromy}     follows     from     Proposition
\ref{monodromyandgassner} and Theorem \ref{purebraidmainth}. \\

The  main section of  the paper  is Section  \ref{gassnersection}.  In
section \ref{gassnersection}, we show that the image of the pure braid
group  $P_{n+1}$ at  a primitive  $d$-th root  of unity  contains many
unipotent   elements.   More   precisely,   the   proof   of   Theorem
\ref{purebraidmainth}  is by  showing that  for $n\geq  2d$  the image
$\Gamma  _n(q)$  contains  an  arithmetic subgroup  of  the  unipotent
radical  of   a  parabolic  $\Q$-  subgroup.   By   using  results  of
Bass-Milnor-Serre  and Tits  (\cite{Ba-Mi-Se},  \cite{Ti}), and  their
extensions  to  other groups  (\cite{Ra},  \cite{Va},  \cite {Ve})  on
unipotent generators  for noncocompact arithmetic  groups of $\R$-rank
at least  two, one can  then show that  such groups are  arithmetic if
$n\geq 2d$. \\

In the  Burau  case, this  was  proved in  \cite{Ve2}. It  was
possible to obtain unipotent elements in the Burau case when $n+1$ was
divisible  by $d$,  since in  that case,  the Burau  representation at
$d$-th roots of unity is {\it degenerate}. The unitary group $U(h)$ of
the relevant  hermitian form is not  reductive and we can  get, in the
image of the Burau representation, elements which lie in the unipotent
radical of $U(h)$. \\

An  analogous result  in the  Gassner case  is proved  in  the present
paper.    We  exploit   the  fact   that  (if   $n\geq  2d$)   then  a
subrepresentation of the restriction  of the Gassner representation at
roots  of unity,  to  a  suitable smaller  pure  braid group,  becomes
degenerate.  One can then  generate unipotent elements.  The existence
of  such  a  suitable  smaller  pure  braid  group  is  ensured  by  a
pigeon-hole argument  if $n\geq  2d$.  This is  worked out  in Section
\ref{gassnersection}. \\

In section \ref{cycliccoverings}  we relate the Gassner representation
to the pure braid action on certain finite index subgroups of the free
group  on $n+1$  generators.  This  relation  is obtained  by using  a
Theorem of  Artin on the action of  the (pure) braid group  on the free
group on $n+1$ generators. We then relate this action to the monodromy
in section \ref{monodromysection}. \\

\newpage

\tableofcontents

\section{Algebraic Groups} \label{alggroups}

The following  theorem is an extension  to all simple  groups, and all
opposing parabolic subgroups, of  a result of Bass-Milnor-Serre and of
Tits (the theorem of Bass-Milnor-Serre  and Tits was proved for $SL_n$
($n  \geq 3$)  and $Sp_{2g}$  ($g\geq  2$)), and  where the  parabolic
subgroup was a  minimal parabolic subgroup.  We refer  to section 2 of
\cite{Ve2}  for a detailed  description and  definitions of  the terms
involved.

\begin{theorem}  \label{bamise} Suppose  $G$ is  an  absolutely almost
simple linear  algebraic group defined  over a number field  $K$, such
that $K$-rank  of $G$ is $\geq  1$ and $G(O_K)$ has  higher real rank,
i.e.
\[\infty  - rank  (G)  \stackrel  {def}{=} \sum  _{v \mid  \infty}
K_v-rank (G) \geq 2.\] Suppose  $P$ is a parabolic $K$-subgroup of $G$
with  unipotent   radical   $U$  and   let  $P^{-}$  be   a  parabolic
$K$-subgroup  defined  over $K$  and  opposed  to  $P$ with  unipotent
radical $U^{-}$.  Let  $\Gamma \subset G(O_K)$  be a  subgroup which
intersects $U(O_K)$ in  a finite index subgroup {\rm  (} and similarly
with $U^{-}(O_K)${\rm )}. Then $\Gamma $ has finite index in $G(O_K)$.
\end{theorem}

\subsection{An  inductive step  for integral  unitary groups}  

In  this subsection,  we prove  a  result which  will be  used in  the
inductive proof of Theorem 1.  The  result says that a subgroup of the
integral unitary  group has finite  index if it contains  finite index
subgroups of  smaller integral unitary  groups.  In the  following, we
will assume that  $E$ is a totally imaginary  quadratic extension of a
totally real number field $K$. Let $x\mapsto {\overline x}$ denote the
action of the non-trivial element of the Galois group of the quadratic
extension $E/K$.   Assume that $V$ is a  finite dimensional $E$-vector
space  and that  $h: V\times  V\ra E$  a $K$  bilinear form  such that
$h(\lambda  v,  \mu w)  =\lambda  {\overline  \mu  } h(v,w)$  for  all
$\lambda,   \mu  \in   E$  and   all   $v,w  \in   V$.   Assume   that
$h(w,v)={\overline  h(v,w)}$ for  all $v,w  \in V$.  Then  the unitary
group $U(h)$ (resp. the special  unitary group $SU(h)$) of elements of
$GL(V)$ (resp.  $SL(V)$) which preserve $h$ is  naturally an algebraic
group (resp. an  almost simple algebraic group) over  the totally real
number field $K$. Under suitable  conditions, we will be able to apply
Theorem \ref{bamise} to $SU(h)$. \\

Moreover,  if  $K_v\simeq \R$  is  an  archimedean  completion of  the
(totally real)  number field  $K$, then the  base change of  $U(h)$ to
$K_v$ is the usual unitary group  of the Hermtitian form over $\R$. In
particular,  the special  unitary group  $SU(h)(K_v)$ is  a co-compact
subgroup  of $U(h)(K_v)$.  As  a consequence,  the  group of  integral
points  $U(h)(O_K)$ and  $SU(h)(O_K)$  are commensurable.   Therefore,
arithmetic  subgroups of  $U(h)$ or  of $SU(h)$,  are the  same  up to
commensurability. \\

We note that  in our applications, the goups  involved will be unitary
groups  of  {\it  skew  hermitian  forms};  but  these  are  naturally
isomorphic to unitary groups of  hermitian forms, by changing the skew
form by a multiple of an imaginary element. For this reason, we do not
stress  the  nature of  the  form, whether  it  is  hermitian or  skew
hermitian.

\begin{notation} With  the preceding  assumptions, Let $V=(V,h)$  be a
nondegenerate hermitian space over $E$ such that $E-rank (V,h)\geq 2$;
since the  special unitary  group is a  $K$-group, this  hypothesis is
equivalent to  $K-rank (SU(h))\geq 2$.  Let $W,\quad  W'$ be codimension
one  subspaces on  which $h$  is again  non-degenerate.   Suppose that
$\Gamma  \subset H_V(O_K)$ is  a subgroup  such that  its intersection
with $U_W(O_K)$  (resp $U_{W'}(O_K)$)  has finite index  in $U_W(O_K)$
(resp.  in $U_{W'}(O_K)$).
\end{notation}

\begin{lemma} \label{inductive}  With the preceding  notation, suppose
that there exists a  non-degenerate subspace $W''$ of the intersection
$W\cap W'$  which contains a  nonzero isotropic vector $v$.   Then the
group $\Gamma $ has finite index in $U_V(O_K)$.
\end{lemma}

\begin{proof} By  the non-degeneracy of  $h$ on $W''\subset  W\cap W'$
the space  $W''$ contains  a vector $v^*$,  also isotropic,  such that
$h(v,v^*)=1$. Write the orthogonal decomposition $V=(Ev + Ev^*) \oplus
X $.  Then $W=(Ev + Ev^*)\oplus X\cap W$ and similarly for $W'$.\\

Consider the filtration 
\[0\subset Ev \subset  E\oplus X \subset Ev\oplus X  \oplus Ev^*=V. \]
Denote the  corresponding Heisenberg group (the  unipotent subgroup of
$U(V)$ which  preserves the  flag and acts  by identity  on successive
quotients),  by   $H(V)$  and   its  integral  points   by  $H(V)(O_K)
=H_V(O_K)$.  The  group  $P\subset  U(V)$ which  preserves  the  above
partial  flag is  a parabolic  subgroup  and $H(V)$  is its  unipotent
radical.  Define  similarly, the smaller Heisenberg  groups $H(W)$ and
$H(W')$ and their integral points $H(W)(O_K)$ and $H(W')(O_K)$. \\

By  assumption, $H(W)\cap  \Gamma$  has finite  index in  $H(W)(O_K)$;
similarly for  $H(W')$.  The  two integral Heisenberg  groups generate
$H(W)(O_K)$ up to finite index, since two distinct vector subspaces of
codimension one,  span the whole space.   We thus find  that $\Gamma $
contains a subgroup of finite  index in the integral unipotent radical
of a parabolic $K$ subgroup. \\

Similarly, we  find a  finite index subgroup  of an  opposite integral
unipotent  radical  in  the  group  $\Gamma$.  Therefore,  by  Theorem
\ref{bamise}    applied   to    $SU(h)$,    $\Gamma\cap   SU(h)$    is
arithmetic. Since  $U(h)(O_K)$ and $SU(h)(O_K)$  are commensurable, it
follows that $\Gamma$ is an arithmetic subgroup of $U(h)(O_K)$.
\end{proof}

\subsection{Groups Generated by Complex Reflections}

The results in this subsection deal with irreducibility of the action
of groups generated  by complex reflections on a  complex vector space
(sometimes ones equipped with a hermitian form).  They are essentially
well  known, but  we need  a version  involving additive  subgroups of
vector groups stable under complex reflections and therefore we record
them here. \\

Let $V$  be an $n$-dimensional vector  space over a field  $K$. We say
that an element $T\in GL(V)$  is a {\it generalised reflection} if the
endomorphism  $T-1$ has  one  dimensional image.   Suppose that  $T_1,
\cdots, T_n$ are generalised reflections such that $(T_i-1)(V)=K \e_i$
and $\{\e_i: 1\leq  i \leq n\}$ form a basis of  $V$.  Assume that for
each $i\leq n-1$, $(T_i-1)(\e_{i+1})=b_i\e_i$ with $b_i\neq 0$. Assume
also that if $i\geq  2$ then $(T_i-1)(\e_{i-1})=a_i\e_i$ with $a_i\neq
0$. \\

\begin{lemma} \label{invariantvector}  [1] Let $V$ and $T_i$  be as in
the preceding and $\Delta $ the group generated by the transformations
$\{T_i;  1\leq i  \leq  n\}$. Denote  by  $V^{\Delta }$  the space  of
vectors  in   $V$  invariant   under  $\Gamma$.   Then   the  quotient
$V/V^{\Delta}$ is an irreducible representation of $\Delta$. \\

[2] If in addition, we assume that $T_i(\e_j)=\e_j$ for all $i,j$ with
$\mid  i-j\mid \geq  2$, then  the  space $V^{\Delta  }$ of  invariant
vectors has dimension at most one.

\end{lemma}

\begin{proof}  Suppose  that  $W\neq  V$  is  a  $\Delta  $  invariant
subspace.  If $\e_j\in W$ for some $j$, by the $T_{j-1}$ invariance of
$W$,  the  vector  $(T_{j-1}-1)(\e_j)$  lies in  $W$.  By  assumption,
$(T_{j-1}-1)\e_j$ is a non-zero multiple of $\e_{j-1}$; therefore, $W$
contains  $\e_{j-1}$.  Similarly,  $(T_{j+1}-1)\e_j  \in W$  and is  a
non-zero multiple  of $\e_{j+1}$  if $j\leq n-1$.   Therefore, if $\e_j$
lies in  $W$ for some  $j$, then $\e_1,  \cdots, \e_n$ lie in  $W$ and
hence $W=V$, a contradiction. \\

Consequently, $W$  does not contain  $\e_j$ for any $j$.  Consider the
image $(T_i-1)W \subset W$. If the image is non-zero, then it consists
of all  multiples of  $\e_i$ and this  is impossible by  the preceding
paragraph. Therefore $(T_i-1)W=0$, which  means that $T_i$ is identity
on  $W$  for   every  $i$.  In  other  words,   $W$  is  contained  in
$V^{\Delta}$. This proves part [1] of the lemma. \\

We  will  now  prove  part   [2],  assuming  (as  in  part  [2])  that
$T_i(\e_j)=\e_j$ if  $\mid i -  j \mid \geq  2$.  Suppose that  $v \in
V^{\Delta  } $  is of  the form  $v=x_2\e_2+\cdots+x_n\e_n$  (i.e. the
coefficient $x_1$ of $\e_1$ is zero). Applying $(T_1-1)$ to $v$ we get
$0=(T_1-1)v= x_2b_1\e_1$ whence $x_2=0$.  Now applying $T_2-1$ to $v$,
we  get $0=(T_2-1)v=  x_3b_2\e_2$.  Therefore  $x_3=0$,  $\cdots$.  An
easy induction now establishes that all the $x_i$ are zero.  Hence the
linear  map  $V^{\Delta }\ra  K$  given  by  $v=\sum _{i=1}^n  x_i\e_i
\mapsto x_1$  (the first coordinate function)  is injective. Therefore
the second part of the lemma follows.

\end{proof}

We  now prove a  version of  Lemma \ref{invariantvector}  for additive
subgroups of  a vector  group stable  under the $T_i$.  Let $A$  be an
integral  domain  and  $\Omega   $  a  field  of  characteristic  zero
containing $A$; suppose there is  an involution of the field $\Omega $
(field  automorphism of  order two)  which  stabilises $A$  and $V$  a
finite dimensional $\Omega $ vector space of dimension $n$ with a {\it
non-degenerate} hermitian  form $h$  with respect to  this involution.
Suppose  that  $\{T_i\in  GL(V):  1\leq  i  \leq  n\}$  preserve  this
hermitian form such that the space  of vectors fixed under $T_i$ is of
codimension one;  then the image of  $T_i-1$ is spanned  by the unique
(up  to scalar multiples)  eigenvector for  $T_i$ with  eigenvalue not
$1$, denote it $\e_i$. \\

We will  assume that the  $\{\e_i: 1\leq i  \leq n\}$ form a  basis of
$V$, and that for each $i$,
\[ T_i(\e_{i+1})=a_i\e_{i+1}+ b_i\e_i~~{\rm with}~~ b_i\neq 0,\] 
\[T_i(\e_{i-1})= c_i \e_i+d_i\e_{i-1}~~{\rm with} ~~c_i\neq 0.\]

Under these assumptions we have the 

\begin{lemma}    \label{complexreflections}   Let    $\Gamma   \subset
U(h)(A)\subset  U(V)$  be  a   subgroup  generated  by  these  complex
reflections $T_i$. Let $W$ be an additive subgroup of the vector group
$V$, such  that $W$  is stable under  the operators $T_i$.  Then there
exists a scalar $\lambda \neq 0$ in the integral domain $A$ such that
\[\lambda \e_1, \cdots, \lambda  \e_n \in W.\] In particular, $\Gamma$
acts irreducibly  on the  vector space $V$;  the representation  is in
fact absolutely irreducible.
\end{lemma}

\begin{proof}  Not all  the images  $(T_i-1)W$ can  be zero;  for that
would mean that all the vectors  $w$ in $W$ are point-wise fixed by all
the  $T_i$;  since distinct  eigenspaces  of  a  unitary operator  are
orthogonal, this means that $w$  is orthogonal to $\e_i$ for each $i$;
therefore, $w=0$ since $h$ is nondegenerate. \\

The  image of $(T_i-1)$  consists of  multiples of  $\e_i$. Therefore,
there exists an integer $i$ such that $W$ contains a multiple $\lambda
_i \e_i$ for  some $\lambda _i \neq 0$. Since $W$  is stable under all
the $T_j$, the equation
\[T_{i-1}(\e_i)=  a_{i-1}\e_i   +  b_{i-1}\e_{i-1},\]  shows   that  a
multiple,  namely $b_{i-1}\lambda _i  \e_{i-1}=\lambda _{i-1}\e_{i-1}$
lies  in $W$,  $\cdots,$  multiples  of $\e_1,  \cdots,  \e_i$ lie  in
$W$. Similarly, the equation
\[T_{i-1}(\e_i)=c_{i+1}(\e_{i+1})+d_{i+1}\e_i,~~{d_{i+1}\neq       0},\]
shows that a nonzero multiple of  $\e_{i+1}$ lies in $W$, $\cdots, $ a
multiple of  $\e_n$ lies  in $W$.  This proves the  first part  of the
lemma.\\ 

The  foregoing  proof also  shows  the  irreducibility  for any  field
$\Omega  $ with  an involution  containing  $A$ in  its fixed  points.
Since  the  fixed field  of  $\Omega $  under  the  involution may  be
embedded  in an  algebraically  closed  field $F$,  and  over $F$  the
unitary group becomes $GL_n(F)$, it follows that the irreducibility is
true  in this  case as  well: the  action of  $\Gamma $  is absolutely
irreducible.
\end{proof}

We will now derive a corollary of Lemma \ref{complexreflections} which
will  be  used  later  in  the  proof of  (part  [3]  of)  Proposition
\ref{irreducible}.    We  will  keep   the  notation   preceding  (and
including)   Lemma  \ref{complexreflections}.  Denote   by  $\Z[\Gamma
](\e_i)$  the additive  subgroup  of  $A^n$ spanned  by  the $\Gamma  $
translates of the vector $\e_i$.

\begin{corollary} \label{fullcorollary} Fix $1\leq  i \leq n$. Let $H_i$
denote a subgroup  of the group $A^*$ of units  of the integral domain
$A$ such that for every $h\in H_i$ there exists an element $\gamma \in
\Gamma $ such that $h\e_i= \gamma (\e_i)$. \\

[1] There exists  a $\lambda =\lambda _i \neq 0$ in  $A$ such that for
every $j$, we have $H_i\lambda \e_j \subset \Z [\Gamma ](\e_i)$. \\

[2] Suppose  for {\rm each} $i$,  $H_i$ is as  in [1]. Let $H$  be the
subgroup  of  $A^*$  generated   by  $H_1,  \cdots  H_n$  {\rm  (}then
$H=H_1\cdots H_n$ is the product{\rm  )}. Then there exists a $\lambda
\neq  0$ such  that for  every  $h\in H$  and every  $j$, the  element
$\lambda  h\e_j$ lies  in the  $\Gamma  $ module  generated by  $\e_1,
\cdots, \e_n$.
\end{corollary}

\begin{proof} An  easy induction  shows that [1]  implies [2].  We now
prove [1]. \\

Fix  $i$. By  Lemma \ref{complexreflections},  there exists  a nonzero
$\lambda \in  A$ such that  $\lambda \e_j \in \Z[\Gamma  ](\e_i)$. Let
$h\in H_i$; by  assumption, there exists $\gamma \in  \Gamma$ such that
$\gamma \e_i= h\e_i$. Therefore, $\lambda (h \e_j)=h (\lambda \e_j)\in
h  (\Z[\Gamma  ](\e_i)=  \Z[\Gamma  ](h\e_i)\subset  \Z[\Gamma  ]\gamma
\e_i=\Z[\Gamma ](\e_i)$.
\end{proof}

\subsection{Some results on algebraic groups}

Let $U\subset SL_n(\C)$ be the unipotent algebraic group consisting of
the set of matrices $u$ of the form
\[u=   \begin{pmatrix}   1   &   x_2 & \cdots & x_n   \\  0   &
1 & \cdots & 0 \\ \cdots & \cdots & \cdots \\
0 & 0 & \cdots & 1\end{pmatrix}.\] 
This is the subgroup which preserves the partial flag
\[0 \subset \C e_1 \subset \C^n,\]
and acts trivially on successive quotients. 

\begin{proposition} \label{Uzariskidense} Let $H\subset SL_n(\C)$ be a
reductive  algebraic subgroup which  contains the  unipotent algebraic
group $U$.  Then $H=SL_n(\C)$.
\end{proposition}

\begin{proof} Denote by $T$ the  group of diagonals in $SL_n$. The Lie
algebra of $SL_n$  splits into eigenspaces for the  action of $T$, and
the eigenvectors  are $E_{ij}$ and $E_{ii}-E_{jj}$  where $E_{ij}$ is,
in the  usual notation, the $n\times  n$ matrix whose  $ij$-th entry is
$1$ and all other entries are zero. Then the Lie algebra $\mathfrak u$
of $U$ is spanned by $E_{1i}$ with $1<i$. \\

Let $\mathfrak  h$ be  the Lie  algebra of $H$.  Write the  direct sum
decomposition $sl_n(\C)={\mathfrak h}\oplus {\mathfrak h}'$ as modules
under the  adjoint action  of $H$  (we use the  assumption that  $H$ is
reductive).  Since $U$ is  unipotent, if ${\mathfrak h}'$ is non-zero,
there exists an $X\in {\mathfrak h}'$,  $X \neq 0$, which is fixed by
$U$.  This means  that  the linear  transformation  $X$ commutes  with
$U$. \\

The centraliser $\mathfrak z$ of $U$ in $sl_n(\C)$ is stable under the
action of the diagonals $T$, since $U$ is $T$-stable. Hence $\mathfrak
z$ splits  into eigenspaces  for $T$. Since  $E_{1i}\in Lie  (U)$, the
equation  $[E_{1i},E_{ij}]=E_{1j}\neq 0$  shows  that the  centraliser
$\mathfrak  z$ cannot  contain $E_{ij}$  with $i\geq  2$.  It  is also
clear that the  lie algebra of $T$ acts faithfully  on $Lie (U)$ under
the adjoint action; therefore,
\[{\mathfrak  z}   =\oplus  \C  _{j\geq  2}   E_{1j}  =Lie  (U)\subset
{\mathfrak  h}.\] Hence  $X$  must  lie in  ${\mathfrak  h}$; this  is
impossible and therefore, ${\mathfrak h}'=0$.

\end{proof}

Let $V$ be an $n$-dimensional vector space over $\C$ and $W,W'$ be two
distinct  codimension  one  subspaces  and  suppose  we  are  given  a
decomposition $V=W\oplus  \C v$ and  $V=W'\oplus \C v'$.   Assume that
$v,v'$  are linearly  independent  over $\C$.   We  will view  $SL(W)$
(resp.   $SL(W')$)  as  the  subgroup  of elements  of  $SL(V)$  which
stabilise  the subspace  $W$  (resp.   $W'$) and  fix  the vector  $v$
(resp. $v'$).

\begin{lemma} $SL(V)$ is generated by $SL(W)$ and $SL(W')$. 
\end{lemma}

\begin{proof}  Put $X=W\cap  W'$.  Then by  our  assumptions, $X$  has
codimension  one  in  both  $W$  and  $W'$.  Fix  a  vector  $w\in  W$
(resp.  $w'\in  W'$)  which  does   not  lie  in  $X$.   We  have  the
decomposition $W=X\oplus  \C w$ and  $W'=X\oplus \C w'$.   Write $E=\C
w\oplus  \C  w'$.    Then  $sl(V)=sl(X)\oplus  (X  \otimes  E^*)\oplus
(X^*\otimes E)\oplus sl(E)  \oplus Y$. Here $Y$ is  the space of trace
zero endomorphisms of $V$ which act by a scalar on $X$ and by a scalar
on $E$. \\

If $h$ is the sub-algebra  generated by $sl(W)$ and $sl(W')$, then $h$
(in  fact  the  subspace  $sl(W)+sl(W')$) contains  $E\otimes  X$  and
$X^*\otimes  E$  as  subspaces;  the sub-algebra  generated  by  these
subspaces contains  $sl(E)$  (it is  easy to capture the  remaining one
dimensional space $Y$ in the subalgebra $h$). Therefore $h=sl(V)$.
\end{proof}

\subsection{Products} The following Lemma  is proved in \cite{Ve2} and
will be used in deducing the arithmeticity of the monodromy in Theorem
\ref{cyclicmonodromy}  from the  arithmeticity  of the  images of  the
Gassner     representation    at     roots    of     unity    (Theorem
\ref{purebraidmainth}). After it was obtained, we learnt that this was
already  proved (in  roughly  the  same form)  in  \cite{Looi2} and  in
\cite{Gru-Lub}. \\

Suppose $X$  is a finite indexing  set and for each  element $p\in X$,
let $K_p$ be  a number field and $G_p$ be  an absolutely almost simple
group defined  over $K_p$ with  $\infty -rank (G_e)\geq 2$.  We assume
that if  $e,f \in X$ are  distinct elements of $X$,  then either $K_e$
and $K_f$ are  not isomorphic as number fields or  $G_e$ and $G_f$ are
not isomorphic  as algebraic groups over $K_e\simeq  K_f$ (both groups
thought of as algebraic groups over the same field $K_e\simeq K_f$).

\begin{lemma} \label{products} With the preceding assumptions, suppose
$\Gamma \subset \prod  _{e\in X} G_e(O_e)$ is a  subgroup of a product
of higher rank  arithmetic groups. Assume that for  each $e\in X$, the
projection of $\Gamma $ in  $G_e(O_e)$ has finite index in $G_e(O_e)$.
Then $\Gamma $ has finite index in the product.
\end{lemma}

\section{Action of the Braid Group on a Free Group}
\label{artinsection}

In this  section, we first  recall an action  of the braid group  on a
free  group, defined  by Artin  (see \ref{artinfree}).   This  gives an
action of the  pure braid group on the first  integral homology of the
commutator subgroup of  the free group. This action  of the pure braid
group is closely related to  the Gassner representation.  To make this
relation  precise, we need  to replace  the free  group with  the free
product $F$  of the  free group with  its abelianisation. There  is an
action of  the braid group on  the kernel of  the natural homomorphism
from $F$  onto the  abelianisation of the  free group.   The resulting
action of the pure braid group on the homology of {\it this} kernel is
identified      with     the      Gassner      representation     (see
\ref{gassnersubsection}).  We  can  then  define the  reduced  Gassner
representation and construct a convenient basis 
$\{\e_i: 1\leq i \leq n\}$ for it; we will
use this basis  in the next section to  specialise the reduced Gassner
representation at roots of unity. \\

\subsection{The pure braid group} \label{purebraidsubsection}
The braid  group $B_{n+1}$ on $n+1$  strands is the free  group on the
generators $s_1,s_2, \cdots, s_n$ modulo the relations
\[s_is_j=s_js_i~~(\mid   i-j\mid  \geq  2),   ~~{\rm  and}~~s_is_js_i=
s_js_is_j~~(\mid i-j\mid =1).\]

The symmetric  group $S_{n+1}$ on $n+1$  symbols is the  free group on
the generators $\sigma_1, \cdots, \sigma_n$ modulo the relations
\[\sigma _i  \sigma _j=\sigma _j\sigma  _i~~(\mid i-j \mid  \geq 2),~~
\sigma_i\sigma _j \sigma _i= \sigma _j \sigma _i \sigma _j ~~(\mid i-j
\mid =1), \] and the additional relations $\sigma_i^2=1$. \\ 

There is a natural surjective homomorphism $B_{n+1}\ra S_{n+1}$ of the
braid group  onto the symmetric group  on $n+1$ letters  given by $s_i
\mapsto \sigma  _i$.  The  kernel of this  homomorphism is  the ``Pure
Braid  Group'' $P_{n+1}$  on $n+1$  strands. The  elements  $s_i^2$ of
$B_{n+1}$  lie  in  $P_{n+1}$.   It  can  easily  be  shown  that  the
conjugates  of  these  elements  $s_i^2$  under all  the  elements  of
$B_{n+1}$  generate $P_{n+1}$.   If $i<j$  denote by  $\Pi  _{ij}$ the
product  in  $B_{n+1}$   given  by  $\Pi  _{ij}=  s_{i+1}s_{i+2}\cdots
s_{j-1}$. For  $r<s$, Set  $A_{rs}=\Pi _{rs}^{-1}s_r^2 \Pi  _{rs}$. In
particular,  $A_{r,r+1}=s_r^2$.  The  pure  braid  group  is  in  fact
generated by the elements $A_{rs}$. \\

\subsection{Artin's   Theorem} \label{artinfree} 

Let $F_{n+1}$ be the free group on $n+1$ generators $x_1, \cdots, x_{n+1}$. 
The  braid  group  $B_{n+1}$  acts  (\cite{Bir},  page  21,  Corollary
(1.8.3)) on  the free group  $F_{n+1}$ as follows.
\[s_i(x_j)=x_j~~ {\rm if} ~j\neq i,~i+1,\] 
\[s_i(x_i)=x_ix_{i+1}x_i^{-1}~~ {\rm and} ~~s_i(x_{i+1})=x_i.\] 

The following Theorem of Artin is fundamental to the rest of the section. 
\begin{theorem} \label{artin} The above formulae give an action of the
braid group on $F_{n+1}$; moreover, the action is faithful.
\end{theorem}

The action of $B_{n+1}$ is  such that on the abelianisation $\Z^{n+1}$
of $F_{n+1}$, the  action is by the symmetric  group $S_{n+1}$ and the
kernel  of  the map  $B_{n+1}\ra  S_{n+1}$  is  the pure  braid  group
$P_{n+1}$. .

The  action  of the  generators  $A_{r,s}$  of  the pure  braid  group
$P_{n+1}$ can  be worked  out (from these  formulae for the  action of
$s_i$) (\cite{Bir}, p. 25, Corollary 1.8.3)):
\[A_{r,s}(x_i)=x_i~~(i<r  ~~{\rm or}  ~~i>s),  \quad A_{r,s}(x_r)=(x_r
x_s)x_r(x_rx_s)^{-1},\]
\[A_{r,s}(x_s)=x_rx_sx_r^{-1},                                    \quad
A_{r,s}(x_i)=[x_r,x_s]x_i[x_r,x_s]^{-1}~~(r<i<s).\]   In   particular,
each  generator $x_i$  of $F_{n+1}$  goes into  a conjugate  of itself
under the action of $P_{n+1}$.

\subsection{Action on Certain Subgroups and Sub-quotients} 
\label{subquotients}

Suppose $F$ is a  group, and $Q$ a quotient of $F$  and $K$ the kernel
of the quotient map $F\ra Q$. Then there is the exact sequence
\[1  \ra K  \ra F\ra  Q \ra  1.\] Denote  by $K^1=  [K,K]$ the
commutator subgroup of $K$ and by $K^{ab}=K/K^1$ the abelianisation
of $K$. Then the conjugation action  of $F$ stabilises  $[K,K]$ and $F$
acts on  $K^{ab}$.  We may  write $K^{ab}$ additively.  The  action of
$F$ on $K^{ab}$  is such that $K$ acts trivially;  hence the action of
$F$ on  $K^{ab}$ descends to  an action of  $Q$ on $K^{ab}$  and hence
$K^{ab}$ becomes a  $\Z[Q]$-module where $\Z[Q]$ is the  group ring of
$Q$ with $\Z$-coefficients. \\

We have an exact sequence 
\[0 \ra K^{ab} \ra F/K^1 \ra Q \ra 1.\] Suppose $H \subset Aut (F)$ be
a subgroup of  the automorphism group of $F$  such that $H$ stabilises
$K$ and  {\it acts trivially} on  $Q$; then $H$ acts  on the foregoing
exact  sequence and the  action of  $H$ on  $K^{ab}$ commutes  with the
action of $Q$ on $K^{ab}$;  therefore, $H$ acts by $\Z[Q]$-module maps
on the $Q$ module $K^{ab}$.

\subsection{The Gassner Representation} \label{gassnersubsection}

In   the  notation   of  subsection   (\ref{subquotients}),   we  take
$F=F_{n+1}*F_{n+1}^{ab}$ to be the free product of the group $F_{n+1}$
and  its  abelianisation  $F_{n+1}^{ab}$  (written  multiplicatively).
Write, temporarily,  $H$ for $F_{n+1}$.  The abelianisation  of $F$ is
$H^{ab}\times   H^{ab}$.   There   is  the   multiplication   map  $m:
H^{ab}\times  H^{ab}$  given  by   $(x,y)\mapsto  xy$.   We  have  the
composite      map     $\phi:     F=H*H^{ab}\ra      H^{ab}     \times
H^{ab}\stackrel{m}{\ra}  H^{ab}$.  This  is a  surjection  with kernel
$K$, say. We then have a {\it split} exact sequence
\[1 \ra K \ra H*H ^{ab} \ra H  ^{ab} \ra 1 .\] The group in the middle
is then  a semi-direct  product $H* H^{ab}\simeq  K \rtimes  H ^{ab}$,
since  the   exact  sequence  splits.   Write  the   elements  of  the
semi-direct  product as  a  pair $(w,t)$  with  $w\in K$  and $t\in  H
^{ab}$.  Write  the image of the  standard generators $x_i$  of $H$ in
this semi-direct  product group as a  pair $x_i=(y_i,X_i)$.  Therefore
$y_i$ and  $X_j$ generate the group  $F$ and hence  the $y_i$ generate
$K$ as a normal subgroup of $F$. \\

As in subsection \ref{subquotients}, we  take the quotient of the group
$F$ by the commutator subgroup $[K,K]$, and get an exact sequence
\[0  \ra  K^{ab}  \ra \frac{H*H ^{ab}}{[K,K]}  \ra  H ^{ab} \ra  1 ,\] 
which is  still split  over $H^{ab}$. Hence  we may
write  the group  in the  middle  as a  semi-direct product  $F/[K,K]=
K^{ab}\rtimes H ^{ab}$. An  element of this group is  written as a
pair  $(w,t)$   with  $w\in  K^{ab}$  and   $t\in  H ^{ab}$;  the
conjugation by  $t$ on $K^{ab}$  is simply multiplying by  the element
$t$,  when  we  view  $K^{ab}$   as  a  module  over  the  group  ring
$\Z[H ^{ab}]$.  We  write $H^{ab}=F_{n+1}^{ab}$ multiplicatively  in the
form  $H^{ab}=X_1^{\Z}X_2^{\Z}\cdots  X_{n+1}^{\Z}$.  Denote  by
$e_i$ the image  of $y_i\in K$ in the  abelianisation $K^{ab}$. By the
conclusion of the last paragraph, the elements $e_i$ generate $K^{ab}$
as a module over the group ring $\Z[H ^{ab}]$:
\[K^{ab}= \sum  _{i=1} ^{n+1} \Z[H ^{ab}](e_i)  . \] We  will now show
that $K^{ab}$  is a free  module over $R$  with $e_i$ as  basis. Write
$R=\Z[H^{ab}]$.  We will view $R$  as a module over the multiplicative
group  $H  ^{ab}$ by  the  formula  $x(f_1,  \cdots, f_{n+1})=  (xf_1,
\cdots, xf_{n+1})$  where $x\in H ^{ab}$  is viewed as a  unit in $R$.
Let $(\xi _i)_{i\leq  n+1}$ be the standard basis  of $R^{n+1}$.  Form
the semi-direct product ${\mathcal H}=R^{n+1}\rtimes H^{ab}$.  We then
get a homomorphism from the  free product $H* H ^{ab}$ into ${\mathcal
H}$ by specifying the  homomorphism on the generators $x_i\mapsto (\xi
_i, X_i)$  and $t\mapsto (0,t)\in  {\mathcal H}=R^{n+1}\rtimes F_{n+1}
^{ab}$ .  Then we get a  homomorphism $H* H^{ab}$ which takes $y_i$ to
the element $\xi _i$. Therefore,  we get a homomorphism of $R$ modules
from $K^{ab}$ into $R^{n+1}$ which  sends $e_i$ into the basis element
$\xi _i$. This shows that the $e_i$ are linearly independent over $R$;
the last line of the preceding  paragraph tells us that the $e_i$ span
$K^{ab}$. Hence the $e_i$ form a basis of $K^{ab}$ and
\[K^{ab}=R^{n+1}=\bigoplus _{i=1}^{n+1} Re_i .\]

We now  write the product  of two elements $x=(w,t),  ~y=(w',t') \in
{\overline F}_{n+1} =K^{ab}\rtimes H^{ab}$. The product is given by
\[xy=(w,t)(w',t')=(w+tw't^{-1},tt')=(w+t(w'),tt').\] 
The inverse of $x=(w,t)$ is $x^{-1}=(-t^{-1}w,t^{-1})$. 
An easy induction shows that 
\[(v_1,t_1)(v_2,t_2)\cdots     (v_{n+1},t_{n+1})=    (\sum
_{i=1}^{n+1}  t_1t_2\cdots  t_{i-1}v_i, t_1t_2\cdots t_{n+1} ) .  \] In  this
formula, $v_i$ are  vectors in $K^{ab}$ and $K^{ab}$ is viewed as an $R$-
module.  The following lemma  is an immediate consequence  of these
formulae and the formulae in \ref{artinfree}.

\begin{lemma}  \label{semidirect}  Let  $x_r,  x_s\in F_{n+1}=H$  be  as
before  and  $(e_r,X_r)={\overline x}_r,  (e_s,X_s)={\overline x}_s
\in {\overline  F}=F/[K,K]$ be their images in  the quotient group
$F/[K,K]$ (which is a semi-direct product).  Then we have the formulae
(read  in  ${\overline  F}$) 
\begin{equation}   \overline   {x_rx_sx_r^{-1}}=   (e_r+X_re_s-X_se_r,
X_s)=((1-X_s)e_r+X_re_s,X_s),\end{equation}
\begin{equation}
x_rx_sx_rx_s^{-1}x_r^{-1}
=(e_r+X_r(e_s)+X_rX_s(e_r)-X_r^2(e_s)-X_re_r,X_r),\end{equation}
\begin{equation}    \label{commutator}    \overline   {[x_r,x_s]}    =
(1-X_s)e_r-(1-X_r)e_s ~~{\rm and}\end{equation}
\begin{equation}\overline{[x_r,x_s]x_i[x_r,x_s]^{-1}}=
(e_i+(1-X_i)v_{r,s},X_i)=\end{equation}
\begin{equation}=(e_i+(1-X_i)(1-X_s)e_r-(1-X_i)(1-X_r)e_s, X_i).\end{equation}
\end{lemma}

The braid group  $B_{n+1}$ acts on the free  group $F_{n+1}$ and hence
acts  naturally  on  the  free product  $F=F_{n+1}*F_{n+1}^{ab}$.  The
preceding   map  $F\ra   F^{ab}=H^{ab}\times   H^{ab}\ra  H^{ab}$   is
equivariant for  the action of  $B_{n+1}$ (and $B_{n+1}$ acts  via the
finite  group $S_{n+1}$ on  the abelianisation  $F_{n+1}^{ab}$). Hence
$B_{n+1}$ acts on the exact sequence
\[1 \ra K \ra F \ra  H^{ab}\ra 1,\] and the pure braid group $P_{n+1}$
acts  trivially on  $F_{n+1}^{ab}=H^{ab}$.   We are  therefore in  the
situation  of  subsection   (\ref{subquotients}),  and  hence,  as  in
subsection   (\ref{subquotients}),  the   group   $P_{n+1}$  acts   by
$\Z[F_{n+1}^{ab}]=R$-module maps on $K^{ab}$.   We can now compute the
action of  the standard generators  $A_{r,s}$ of the pure  braid group
$P_{n+1}$ on the images of $x_i$ in the quotient group $F/[K,K]$. \\

[1]  Recall  that  $A_{r,s}(x_i)=x_i$   if  $i\leq  r-1$  or  $i  \geq
s+1$. From Lemma  \ref{semidirect}, it follows that $A_{r,s}(e_i)=e_i$
for these $i$. \\

[2] $A_{r,s}(x_r)=^{x_rx_s}(x_r)$.  When this equation  is read modulo
$[K,K]$, we see from Lemma \ref{semidirect} that
\[A_{r,s}(e_r)X_r                         =A_{r,s}(e_rX_r)=A_{r,s}(x_r)
=e_rX_re_sX_s(e_rX_r)X_s^{-1}e_s^{-1}X_r^{-1}e_r^{-1} .\]
   
Cancelling $X_r$ on the right on the left most and right most sides of
this equation, we see that
\[A_{r,s}(e_r)= (1-X_r+X_rX_s)e_r+X_r(1-X_r)e_s .\]

[3]  The  equation   $A_{r,s}(x_s)=^{x_r}(x_s)$  becomes,  modulo  the
subgroup $[K,K]$, the equation
\[A_{r,s}(e_s)X_s= ((1-X_s)e_r+X_re_s)X_s.\]    

[4] If $r<i<s$ then by Lemma \ref{semidirect},  
\[A_{r,s}(e_i)X_i=A_{r,s}(x_i)=                      ^{[x_r,x_s]}(x_i)=
(e_i+(1-X_i)((1-X_s)e_r-(1-X_r)e_s)) X_i \] or
\[A_{r,s}(e_i)= e_i +(1-X_i)((1-X_s)e_r-(1-X_r)e_s). \]

These  equations  imply  that  with  respect to  the  basis  $e_i$  of
$K^{ab}=R^{n+1}$, the  action by $P_{n+1}$ on the  $R$ module $K^{ab}$
is exactly the {\bf Gassner representation} $G_n(X): P_{n+1}\ra GL_{n+1}(R)$
(see \cite{Bir}, p 119, formulae (3-24)). \\

\begin{notation} The  ring $R=\Z[X_1^{\pm 1}, \cdots, X_{n+1}^{\pm
1}]$  of   Laurent  polynomials  in  $n+1$   variables  with  integral
coefficients,  is an integral  domain. Let  $\Omega =\Q  (X_1, \cdots,
X_{n+1})$  be its  field of  fractions.  We  have the  free $R$-module
$K^{ab}=\sum _{i=1}^{n+1}  R e_i$.  This may be  thought of as  an $R$
submodule of the $\Omega  $ vector space $K^{ab}\otimes _R \Omega=\sum
_{i=1}^{n+1}\Omega e_i$. We may write $e_i=(1-X_i)v_i$ for some vector
$v_i\in K^{ab}\otimes \Omega$. \\

Now  the full  braid  group $B_{n+1}$  acts  on $R$  via $S_{n+1}$  by
permuting the indices $X_i$ of the generators (and the pure braid group
acts  trivially).   Hence  $B_{n+1}$  acts  on  $\Omega   $  by  field
automorphisms. We denote this  action, for $g\in B_{n+1}$ and $\lambda
\in  \Omega  $, by  $(g,\lambda  )\mapsto  g(\lambda)$.  The action  of
$B_{n+1}$ on $K^{ab}\otimes  \Omega$ is not linear over  $\Omega $ but
is ``twisted  linear'': if $\lambda  \in \Omega $, $g\in  B_{n+1}$ and
$w\in K^{ab}$, then $g(\lambda w)=g(\lambda ) g(w)$.

\end{notation}

\begin{lemma} \label{vbasis} Let $v_i=\frac{1}{1-X_i}e_i$ with $e_i$
and $R$ as before. The $R$ module $\bigoplus _{i=1}^{n+1} Rv_i$ is stable
under the action of $P_{n+1}$.  
\end{lemma}

\begin{proof}  Since  $P_{n+1}$  acts  by $\Omega  $-linear  maps,  it
suffices to  show that  for every $g\in  P_{n+1}$ and every  $v_i$ the
translate $g(v_i)$ is an $R$  linear combination of the $v_j$. We will
in fact prove more; we will show that for every generator $s_i$ of the
{\it full braid  group} $B_{n+1}$, the translate $s_i(v_j)$  is an $R$
linear combination of  the $v_k$. The group $B_{n+1}$  acts by twisted
$\Omega $  linear maps  as before  and not by  $\Omega $  linear maps;
however, it takes an element  $\lambda w\in K^{ab}$, with $\lambda \in
R$ and $w\in K^{ab}$ into an  element of the form $\mu g(w)$ and hence
preserves the space $\sum R v_i$ provided each $v_j$ is mapped into an
$R$ linear combination  of the $v_k$. We now need  only check that for
each  generator  $s_i$ of  $B_{n+1}$  and  each  $v_j$, the  translate
$s_i(v_j)$ is an  $R$ linear combination of the  vectors $v_1, \cdots,
v_{n+1}$. \\

Suppose   $j\neq   i,i+1$.    We   have  $s_i(x_j)=x_j$   for   $j\neq
i,i+1$. Therefore,  $s_i(X_j)=X_j$. Since $x_i=(e_i,  X_i)$ it follows
that $s_i(e_j)=e_j$. We now write $e_j=(1-X_j)v_j$ and note that $s_i$
acts trivially on $X_j$. Hence
\[(1-X_j)v_j=e_j=                              s_i(e_j)=s_i(1-X_j)v_j)=\]
\[=(1-s_i(X_j))s_i((v_j)=(1-X_j)s_i(v_j).\]      This      shows     that
$s_i(v_j)=v_j$. \\

Suppose $j=i$.  Then $s_i(x_i)=x_ix_{i+1}x_i ^{-1}=[x_i,x_{i+1}]x_{i+1}
$. We  have expressed a commutator in  terms of the $e_j$  (see (4) of
Lemma   \ref{semidirect}):   hence   the   commutator   $[x_i,x_{i+1}]=
(1-X_{i+1})e_i-(1-X_i)e_{i+1}$. Therefore,
\[(s_i(e_i),            X_{i+1})=           s_i((e_i,           X_i))=
((1-X_{i+1})e_i-(1-X_i)e_{i+1}+e_{i+1},X_{i+1}).    \]   Comparing  the
extreme  left and  right  hand sides  of  this equation,  we see  that
$s_i(e_i)= (1-X_{i+1})e_i+X_ie_{i+1}$.  Now write $e_i=(1-X_i)v_i$ and
similarly for $e_{i+1}$. Then we have
\[(1-X_{i+1})s_i(v_j)=s_i((1-X_i)v_i)=s_i(e_i)= (1-X_{i+1})e_i+X_ie_{i+1}=\]
\[=(1-X_{i+1})(1-X_i)v_i+X_i(1-X_{i+1})v_{i+1}.\]            Cancelling
$(1-X_{i+1})$ on  both the extreme right  and left hand  sides of this
equation, we get
\[s_i(v_i)= (1-X_i)v_i+X_iv_{i+1}.\] 

Suppose  $j=i+1$. Then  $s_i(x_{i+1})=x_i$. Reading  this as  in Lemma
\ref{semidirect}   we   get   $s_i(e_{i+1},X_{i+1})=(e_i,  X_i)$   and
$s_i(X_{i+1})=X_i$.  Therefore,  we  get  $s_i(e_{i+1})=e_i$.  Writing
$e_i(1-X_i)v_i$ we see that
\[(1-X_i)s_i(v_{i+1})= s_i((1-X_{i+1})v_{i+1})=s_i(e_{i+1})= e_i=(1-X_i)v_i.\]
Comparing the extreme right and left hand sides of this equation, we get
\[s_i(v_{i+1})= v_i.\]  From the last  three paragraphs, we  see that
each $s_i(v_j)$ is an  $R$-linear combination of the $v_j$. Therefore,
the lemma follows.
\end{proof}

\begin{lemma} \label{epsilonbasis} Set $\e_i=v_i-v_{i+1}$ for $1\leq i
\leq n$.Then the $R$ module generated by $\{\e_i: 1\leq i \leq n\}$ is
stable under the action of  $P_{n+1}$.  \\ 

In  particular, with  respect to  the basis  $\e_1, \cdots,  \e_n$ the
transformation $T_i=s_i^2$ has the matrix form
\[\begin{pmatrix} 1 & 0 & 0 \\ X_i (1-X_{i+1}) & X_iX_{i+1} & 1-X_i \\ 0
& 0 & 1 \end{pmatrix} \oplus 1_{n-3} , \] where the $3\times 3$ matrix
is with respect  to the basis elements $\e_{i-1},  \e_i, \e_{i+1}$ and
$s_i^2$ acts as identity on the basis elements $\e_j$ for the
other indices $j$. In particular, $s_i^2$ are complex reflections. 
\end{lemma}

\begin{proof} As in the proof  of Lemma \ref{vbasis}, because the full
braid group acts by ``twisted'' $\Omega$ linear maps on $K^{ab}\otimes
\Omega$, it suffices to check  that the full braid group preserves the
$R$ module spanned by the $\e_i$. \\

We  use the  formulae in  the  proof of  Lemma \ref{vbasis}.   Suppose
$j+1<i$. Then
\[s_i(\e_j)=s_i(v_j-v_{j+1})=v_j-v_{j+1}=\e_j.\]             Similarly,
$s_i(\e_j)=\e_j$  if  $j>i+1$.  We  now  get  from  the  formulae  for
$s_i(v_i)$ and $s_i(v_{i+1})$ obtained from Lemma \ref{vbasis}, that
\[s_i(\e_i)= s_i(v_i-v_{i+1})= (1-X_i)v_i+X_iv_{i+1}- v_i=-X_i\e_i,\]
\[s_i(\e_{i+1})=s_i(v_{i+1}-v_{i+2})= v_i-v_{i+2}= \e_i+\e_{i+1}.\]
Finally, 
\[s_i(\e_{i-1})=s_i(v_{i-1}-v_i)= v_{i-1}- (1-X_i)v_i-X_iv_{i+1}=\]
\[=\e_{i-1}+X_i\e_i.\]

This proves the first part of the lemma. \\

To prove the second part, we must compute $s_i^2(\e_j)$. If $j+1<i$ or
if  $j>i+1$ then  $s_i(\e_j)=\e_j$; hence  $s_i^2(\e_j)=\e_j$.  We now
compute $s_i^2(\e_{i-1})$: 
\[s_i(s_i(\e_{i-1}))=s_i(\e_{i-1}+X_i \e_i)= s_i(\e_{i-1})+ X_{i+1}s_i(\e_i)=\]
\[=\e_{i-1}+X_i\e_i+   X_{i+1}(-X_i\e_i)=\e_{i-1}+X_i(1-X_{i+1})\e_i.\]
Next        we       compute        
\[s_i^2(\e_i)=       s_i(-X_i\e_i)=
-X_{i+1}s_i(\e_i)=X_{i+1}X_i\e_i.\]  Finally,
\[s_i^2(\e_{i+1})=s_i(\e_i+\e_{i+1})= -X_i\e_i+\e_i+\e_{i+1}=(1-X_i)\e_i+\e_{i+1}.\]
\end{proof}

\subsection{An invariant element in the Gassner representation} 

The  product element  $x_1x_2\cdots x_{n+1}\in  F_{n+1}$  is invariant
under  the  action  of  the  braid  group  $B_{n+1}$.   The  image  of
$x_1x_2\cdots x_{n+1}$  in the semi-direct product  group $F/[K,K]$ is
therefore  invariant;   since  the  image  of  $x_i$   is  written  as
$(e_i,X_i)$,   it    follows   from   the    formulae   before   Lemma
\ref{semidirect} that
\[x_1x_2\cdots x_{n+1}= (e_1X_1)(e_2X_2)\cdots(e_{n+1}X_{n+1})= \]
\[=(e_1+X_1(e_2)+X_1X_2(e_3)+\cdots+                        X_1X_2\cdots
X_n(e_{n+1}), X_1X_2\cdots X_{n+1}). \] Therefore, the element
\begin{equation}    \label{invariantelement}    v=\sum    _{i=1}^{n+1}
X_1X_2\cdots X_{i-1}(e_i) \in K^{ab} \end{equation} is invariant under
the  action  of  the  pure  braid  group.   Moreover,  since  all  the
coefficients $X_1X_2\cdots X_i$ of $v$ are units in the ring
\[R=\Z[X_1^{\pm 1},X_2^{\pm 1},  \cdots X_{n+1}^{\pm 1}],\] it follows
that  $v$ is  part of  a  basis of  $K^{ab}=\oplus _{i=1}^{n+1}  Re_i=
R^{n+1}$.   Consider the  quotient module  $V_n(X)=  K^{ab}/Rv$.  Then
$V_n(X)\simeq R^n$ and is a module over $P_{n+1}$. Let $\Omega$ be the
field of fractions  of the integral domain $R$.  Then $V_n(X)\otimes _R
\Omega$  is  called  the  {\bf reduced  Gassner  Representation}  over
$\Omega$ the field of fractions, and is denoted
\[g_n(X): P_{n+1}\ra GL_n(\Omega ).\]

\subsection{A supplement to the space of invariants} \label{supplement}

We will now find a sub-module $W_n(X)$ which has zero intersection with
the  space  $Rv$ of  multiples  of the  invariant  vector  $v$ in  the
(non-reduced) Gassner representation, which is stable under the action
of the  pure braid group $P_{n+1}$ and is  free of rank $n$  over $R$. We
will  view $K^{ab}=\oplus  _{i=1}^ {n+1}  Re_i$ as  a subgroup  of the
$\Omega $ vector space $\Omega ^{n+1}=K^{ab}\otimes _R \Omega = \oplus
_{i=1}^{n+1}  \Omega e_i$. Write  $e_i=(1-X_i)v_i$, with  $v\in \Omega
^{n+1}$. \\ 

The calculations  of  Lemma \ref{semidirect}  show that  the
$R$-module $L=\oplus _{i=1}^{n+1} Rv_i$  is stable under the action of
the pure braid  group $P_{n+1}$. The lemma also  implies that the free
$R$ sub-module
\[W=\oplus _{i=1}^n R  \e_i\simeq R^n,\] (where $\e_i=v_i-v_{i+1}$) is
stable under  $P_{n+1}$. We note that  the commutator $[x_i,x_{i+1}]$,
(see equation \ref{commutator} of Lemma \ref{semidirect}) viewed as an
element of the kernel $K^{ab}$ has the form
\begin{equation} 
[x_i,x_{i+1}]= (1-X_i)e_{i+1}-(1-X_{i+1})e_i= (1-X_i)(1-X_{i+1})(v_i-v_{i+1})= 
\end{equation}
i.e. 
\begin{equation} \label{epsilon}
[x_i,x_{i+1}]= (1-X_i)(1-X_{i+1})\e_i. 
\end{equation} As an  $R$ module, $W$ is a  summand of $L$: $L=W\oplus
Rv_{n+1}$ (since  $\e_i=v_i-v_{i+1}$ form a  basis of $W$). It  can be
proved  that $W$  is  {\it not}  a  direct summand  as  a module  over
$P_{n+1}$; however, it  is so, when all the  modules are tensored with
the field $\Omega$:
\[L\otimes \Omega  =\Omega ^{n+1}= W\otimes \Omega  \oplus \Omega v,\]
where   $v=\sum  _{i=1}^{n+1}   X_1\cdots  X_{i-1}e_i$.    Write  $\pi
=X_1\cdots X_i$  and $e_i=(1-X_i)v_i$  as before.  Then  the invariant
element  $v$ of equation  \ref{invariantelement} can  be written  as a
linear combination of $\e_i$ and $e_{n+1}$:
\begin{equation}     \label{invariant}     v=(1-\pi     _1)\e_1+(1-\pi
_2)\e_2+\cdots    (1-\pi    _n)\e_n+    (1-\pi   _{n+1})v_{n+1}    \in
L.\end{equation}  We  will refer  to  the  $P_{n+1}$ module  $W=\oplus
R\e_i\simeq R^n$ as the  {\bf reduced Gassner representation} over the
ring $R$. We will later  consider the reduction (modulo ideals of $R$)
of   $W$   to   obtain   specialisations  of   the   reduced   Gassner
representation.  Over the fraction field $\Omega $, the representation
$W\otimes \Omega $ is isomorphic to the quotient $K^{ab}\otimes \Omega
/\Omega v$ and hence the terminology is consistent with the end of the
preceding subsection.

\section{Properties of the The Gassner Representation}
\label{gassnersection}

In   this  section,   we  prove   some  properties   of   the  Gassner
representation and  its specialisations. We first  {\it construct} (in
subsection  \ref{gassnerhermitian})  a  skew  Hermitian  form  on  the
reduced Gassner representation, which is invariant under the action of
the pure braid group. The existence  of the form is due to \cite{Long}
(see also \cite{Del-Mos}) , but  the construction for the basis $\e_i$
defined in the previous section, may perhaps be new.  Using this form,
it   is  easy   to  decide   when  specialisations   of   the  Gassner
representation (especially $d$-th roots of unity) are irreducible. \\

It turns  out that the  specialised reduced Gassner  representation is
irreducible  if and  only if  the form  is nondegenerate.  When  it is
degenerate, we will  get, in the image of  the Gassner representation,
many             unipotent            elements            (Proposition
\ref{gassnerspecialisation}).  This   is  crucial  to   our  proof  of
arithmeticity  (Theorem \ref{purebraidmainth}).

\subsection{A  (skew)  hermitian  form  preserved by  the  pure  braid
group} \label{gassnerhermitian}

It  is  known  (see   \cite{Long}  Theorem  (3.3))  that  the  Gassner
representation has a skew-hermitian form invariant under the action of
the  pure braid  group.   To compute  this  form with  respect to  the
``$\e$''-basis  $\{\e_i: 1\leq  i  \leq n\}$  of  the reduced  Gassner
representation, we proceed as follows.  Since the matrices $s_i^2$ act
by  complex  reflections, and  $\e_i$  is  the  unique (up  to  scalar
multiples) eigenvector with eigenvalue $X_iX_{i+1} \neq 1$, it follows
from Lemma  \ref{complexreflections} that  the group generated  by the
$s_i^2$  acts  already  irreducibly  on $R^n\otimes  \Omega  $,  where
$\Omega $  is the field of  fractions of $R$.   Further, since $s_i^2$
are  unitary, the  eigenvectors with  eigenvalue $=1$  of  $s_i^2$ are
orthogonal with respect to $h$, to the eigenvector $\e _i$.  Hence $\e
_j$ and $\e _i$ are orthogonal if $\mid i-j\mid \geq 2$. \\

The  ring $R$ has  an involution  given by  $X_i\mapsto X_i^{-1}=Y_i$.
Since our form  is to be skew hermitian, we  have that $h(\e_1, \e_1)$
is  an element  of $R$  which  is ``imaginary''  (i.e. it  goes to  its
negative under the involution). We normalise it so that
\[h(\e_1,    \e_1)=\frac{1-X_1X_2}{(1-X_1)(1-X_2)}.\]   Consider   the
element  $h(\e_1,\e_2)$;  the  invariance  of  $h$  under  forces  the
equation
\[h(\e_1,\e_2)=h(s_1^2(\e_1), s_1^2(\e_2)).\]
Since $s_1^2(\e_1)=X_1X_2\e_1$, and $s_1^2(\e_2)= (1-X_1)\e_1+\e_2$, the 
invariance of $h$ and the chosen value of $h(\e_1, \e _1)$ imply 
\[h(\e_1,\e_2)=     \frac{-X_2}{1-X_2}    ,     \quad    h(\e_2,\e_2)=
\frac{1-X_2X_3}{(1-X_2)(1-X_3)}.   \]  We  can   proceed  in   a  like
manner.  Thus  the  $n\times  n$-matrix $h(X)=(h_{ij})$  of  the  skew
hermitian form is
\[h(X)=      \begin{pmatrix}     \frac{1-X_1X_2}{(1-X_1)(1-X_2)}     &
-\frac{1}{1-X_2}  &  0  &  0  &  \cdots &  0  \cr  -\frac{X_2}{1-X_2}  &
\frac{1-X_2X_3}{(1-X_2)(1-X_3)} & -\frac{1}{1-X_3} &  0 &\cdots & 0 \cr
0    &   -\frac{X_3}{1-X_3}    &    \frac{1-X_3X_4}{(1-X_3)(1-X_4)}   &
-\frac{1}{1-X_4} & 0  & \cdots \cr \cdots & \cdots &  \cdots & \cdots &
\cdots \cr
\end{pmatrix} .\] Thus  the invariance of $h$ implies  that $h$ is the
above  matrix (once  the value  of  $h(\e_1, \e_1)$  is normalised  as
above). \\

The natural  formula for $h(X)$ given  above takes values  in the ring
$R'$  which is  generated  by $R$  together  with the  inverse of  the
element $\prod _{i=1}^{n+1} (1-X_i)$.   One can clear denominators and
ensure   that  $h(X)$   takes   values  in   $R$.  Therefore   $g_n(X)
(P_{n+1})\subset  U(h(X))(S)$,  where  $S$  is  the  sub-ring  of  $R$
invariant under  the involution $f\mapsto {\overline f}$  on $R$ given
by  $X_i\mapsto X_i^{-1}$.   The  unitary group  $U(h)$  is an  affine
algebraic group scheme defined over $S$:
\[U(h)(S)=\{g\in GL_n(R): h(gv,gw)=h(v,w)~~\forall v,w\in R^n\}.\]

The following result is due to \cite{Abd2} (the skew hermitian form 
in \cite{Abd2} is for a different basis) :

\begin{lemma}  \label{abdul}  {\rm [1]}  The  skew  hermitian form  on
$(R^n)$ defined by the matrix  $h(X)$ is invariant under the action of
the  pure braid group.  Therefore, $\rho  _X(P_{n+1})\subset U(h)(S)$,
where  $S\subset R$  is the  sub-ring of  elements invariant  under the
involution. \\

{\rm [2]} The matrix $h(X)$ has determinant
\[        \frac{1-X_1X_2\cdots        X_nX_{n+1}}{(1-X_1)(1-X_2)\cdots
(1-X_n)(1-X_{n+1})} .\]

{\rm   [3]}  In   particular,   the  skew   Hermitian   form  $h$   is
non-degenerate.    
\end{lemma}

\begin{proof} The  form $h$ was constructed under  the assumption that
it was invariant under $P_{n+1}$. Hence we need only prove part [2] of
the lemma.  Part [2] of the lemma is proved by induction. Put $X=(X_1,
X')$  where $X'$  is  the $n$-tuple  $(X_2,  \cdots, X_{n+1})$;  write
$X'=(X_2,  X'')$ where $X''$  is the  $n-1$-tuple $(X_3,  X_4, \cdots,
X_{n+1})$.  Expand   the  determinant   of  the  $n\times   n$  matrix
$h_n(X)=h(X)$  by the  first row;  then $h_n(X)$  may be  expressed in
terms of $h_{n-1}(X')$ and $h_{n-2}(X'')$:
\[h_n(X)         =         \frac{1-X_1X_2}{(1-X_1)(1-X_2)}h_{n-1}(X')-
\frac{X_2}{(1-X_2)^2}h_{n-2}(X'').\] Now induction  on $n$ implies the
formula for the determinant of $h_n(X)$.
\end{proof}

Suppose that  ${\mathfrak a}\subset R$  is a non-zero  ideal invariant
under  the  involution $f\mapsto  {\overline  f}$  on  $R$.  Then  the
quotient map $R\ra R/{\mathfrak a}$ induces a homomorphism $GL_n(R)\ra
GL_n(R/{\mathfrak a})$; the skew hermitian form $h$ descends to a skew
hermitian form $h_{\mathfrak a}$ on the quotient module $(R/{\mathfrak
a})^n$ and hence we have a homomorphism $U(h)(S)\ra U(h_{\mathfrak a})
(S/{\mathfrak a}\cap  S)$; therefore  we have a  representation $g_{n,
{\mathfrak   a}}:P_{n+1}\ra  U(h_{\mathfrak   a})(S/{\mathfrak  a}\cap
S)$.\\

Now consider  a homomorphism from  the ring $R=\Z[X_i^{\pm  1};1\leq i
\leq n+1]$ into the ring $\Z[\omega _d]$ of integers in the cyclotomic
extension $E_d=  \Q(e^{\frac{2\pi i}{d}})=\Q(\omega _d)$  ($\omega _d$
is the  primitive $d$-th root of unity  $e^{\frac{2\pi i}{d}}$).  This
homomorphism is given by $X_i\mapsto t_i$ where $t_i=\omega _d ^{k_i}$
is a $d$-th  root of unity.  Since all the $k_i$  are co-prime to $d$,
it follows  that the group  generated by each  of the $t_i$ is  all of
$\Z/d\Z$, the group  of $d$-th root of unity.   Under the homomorphism
$R\ra \Z[\omega  _d]$, the  sub-ring $S$ maps  into the ring  $O_d$ of
integers in the totally real sub-field $\Q(2{\rm cos}~\frac{2\pi}{d})$
of  the   cyclotomic  field  $E_d$.   Hence  we   have  the  composite
representation, denoted
\[g_n(k,d):  P_{n+1}\ra  U(h)(S)\ra  U(h)(O_d),\]  where  $k$  is  the
$(n+1)$-tuple of integers ($k_1,k_2, \cdots, k_{n+1}$).  \\

The  following is the  main result  of the  paper, from  which Theorem
\ref{cyclicmonodromy} will be deduced.

\begin{theorem} \label{purebraidmainth} Suppose $d \geq 3$, $n\geq 2d$
and  all the  integers  $k_i$ are  co-prime  to $d$.   Then the  image
$\Gamma _n= \Gamma =g_n(k,d) (P_{n+1})$ (of the Gassner representation
$g_n(k,d)$ at primitive $d$-th roots of unity) is a subgroup of finite
index in the integral unitary  group $U(h)(O_d)$.  In other words, the
``monodromy group'' $g_n(k,d)(P_{n+1})$ is an arithmetic group.
\end{theorem}

The full braid  group $B_{n+1}$ has a representation,  called {\it the
reduced Burau representation} (\cite{Bir}, p.118, Example 3),
\[\rho _n(q):  B_{n+1}\ra GL_n(\Z[q,q^{-1}]).\] Since  the restriction
to $P_{n+1}$  of the reduced Burau representation  at primitive $d$-th
roots  of  unity  is  the reduced  Gassner  representation  $g_n(k,d)$
evaluated at primitive  $d$-th roots of unity (when  all the $k_i$ are
equal  to $1$), the  following Theorem  is a  special case  of Theorem
\ref{purebraidmainth}.

\begin{theorem} \label{buraumainth}  If $d\geq  3$ and $n\geq  2d$ and
all the $k_i$ are $1$ , then the image of the Burau representation
\[\rho _n(d):B_{n+1}\ra U(h)(O_d)\]  evaluated at all primitive $d$-th
roots of  unity, is a  subgroup of finite  index. In other  words, the
monodromy group $\rho _n(d)(B_{n+1})$ is an arithmetic group.
\end{theorem}

Theorem   \ref{buraumainth}  was  proved   in  \cite{Ve2},   by  using
properties of the Burau representation at roots of unity, and by using
induction    for   all    $n\geq   2d$.    The   proof    of   Theorem
\ref{purebraidmainth} is similar, and we use properties of the reduced
Gassner  representations  at  roots  of unity.  These  properties  are
essentially well known,  but we need a precise  form of these results.
Theorem  \ref{purebraidmainth}  will be  proved  at  the  end of  this
section, after many preliminary results.

\subsection{Irreducibility}  As we have  seen before,  the reduced
Gassner  representation has a  nondegenerate invariant  skew hermitian
form $h$ with  values in the field of fractions $\Omega  $ of the ring
$R=\Z[X_1^{\pm 1},  \cdots, X_{n+1}^{\pm 1}]$  of Laurent polynomials,
which  is   preserved  by  the  group  $P_{n+1}$   under  the  Gassner
representation. This was determined  on the basis $\e_i$ in subsection
\ref{gassnerhermitian}. By  Lemma \ref{epsilonbasis}, it  follows that
the elements $s_i^2$ are complex reflections.

\begin{proposition}  \label{irreducible}  

{\rm  [1]} If  $R=\Z[X_1^{\pm  1}, X_2^{\pm  1}, \cdots,  X_{n+1}^{\pm
1}]$, and  $\Omega $ is its  quotient field, then  the reduced Gassner
representation
\[G_n(X): P_{n+1}\ra GL_n(R)\subset GL_n(\Omega ),\]
is irreducible. \\

{\rm [2]} The central element $\Delta ^2\in P_{n+1}$ where 
\[\Delta  =\Delta _n  =(s_1s_2\cdots      s_n)(s_1s_2\cdots     s_{n-1})\cdots
(s_1s_2)(s_1),\]  acts by multiplication  by the  scalar $X_1X_2\cdots
X_{n+1}$ on the reduced Gassner representation. \\

{\rm [3]} If  $W\subset R^n$ is an additive  subgroup stable under the
action  of $P_{n+1}$  then there  exists a  scalar $\lambda \in  R$  
such that $\lambda (R^n)\subset W$.
\end{proposition}

\begin{proof}  

[1] At  the beginning  of this subsection,  we have verified  that the
conditions  of Lemma  \ref{complexreflections} are  satisfied.  By the
result   of  \cite{Long}   quoted  earlier,   $\Omega  ^n$   admits  a
nondegenerate     hermitian     form.      Therefore,     by     Lemma
\ref{complexreflections},  the representation  $g_n(X)$  is absolutely
irreducible. \\

[2] The element $\Delta ^2$ is  central in $P_{n+1}$; by part [1], the
central  element acts by  a scalar,  call it  $\lambda $.   Write $\pi
_{n+1}$ for the product  $X_1X_2\cdots X_{n+1}$. We compute the scalar
$\lambda  $  by finding  the  effect of  $\Delta  ^2$  on the  element
$e_1$. Consider  the element $x_1\in F_{n+1}$. A  calculation (see the
formulae  for the  action  of $s_i$  on  the free  group $F_{n+1}$  in
subsection (\ref{artinfree})), shows that the action of $\Delta ^2$ on
$x_1$ is given by (notation: in a group, $^y(x)=yxy^{-1}$)
\[\Delta  ^2   (x_1)=  ^{x_1x_2\cdots   x_{n+1}}  (x_1)  .\]   In  the
semi-direct product $K^{ab}\rtimes F_{n+1}$, the element $x_i$ maps to
${\overline x_i}=(e_i,X_i)=e_iX_i$.  Hence this equation then becomes
\[\Delta ^2 (e_i,X_i)= \Delta   ^2  (e_1X_1)=   
(v\pi   _{n+1})(e_1X_1)\pi  _{n+1}^{-1}v^{-1}=\]
\[=  (\pi_{n+1}e_1+(1-X_1)v, X_1)\] Comparing the vector parts, we get 
\[\Delta ^2  (e_1)=X_1X_2\cdots X_{n+1}(e_1)+(1-X_1)v.\] This is
in the  Gassner representation space  $K^{ab}$. Going modulo  the line
through $v$, we see that
\[\Delta ^2 (e_1)=\pi  _{n+1}e_1=X_1X_2\cdots X_{n+1}e_1 \quad {\rm (}
mod ~~Rv {\rm  )}\] in the reduced Gassner  representation.  This proves
part [2] of the Proposition.\\

We will now prove  part [3]. Fix $i$ with $1\leq i  \leq n$. The group
$P_{i+1}$  operates on  the  $R$ module  generated  by $\e_1,  \cdots,
\e_i$. This  module is nothing but the  reduced Gassner representation
$g_i(X)$. Consequently,  the element  $C_i=(\Delta _i^2)$ acts  by the
scalar   $X_1\cdots   X_{i+1}$   on   the   vectors   $\e_1,   \cdots,
\e_i$. Moreover, $s_i^2(\e_i)=  X_iX_{i+1}(\e_i)$.  In particular, the
group $H_i$  generated by  the elements $X_1\cdots  X_{i+1}, X_1\cdots
X_{i+2}, \cdots, X_1\cdots X_{n+1};  X_iX_{i+1}$ has the property that
the $\Gamma $ module generated  by $\e_i$ contains all elements of the
form $h  (\e_i)$ for  every $h\in  H_i$. Note that  $H_i$ is  also the
group generated  by the elements $X_1\cdots  X_{i+1}, X_{i+2}, \cdots,
X_{n+1},  X_iX_{i+1}$  (successive  ratios  of  the  previous  set  of
elements, together with the first and the last one of the previous set
of elements). \\

Thus  the group $H$  generated by  $H_1, \cdots, H_n$ is  the group
generated by $X_1X_2, X_3,  \cdots, X_{n+1}$ (contribution from $i=1$)
and  $X_2X_3$ (contribution  from  $i=2$). This  is  clearly the  group
generated by $X_1, \cdots, X_{n+1}$. By Corollary \ref{fullcorollary},
this  means that  there exists  a $\lambda  \neq 0$  in $R$  such that
$\lambda  X_1^{\Z}\cdots X_{n+1} ^{\Z}  (\e_i) $  lies in  the $\Gamma
$-module generated  by all  the $\e_i$. Since  monomials in  the $X_i$
generate $R$, this means that $\lambda (R\e_j) \subset \Z[\Gamma] \e_i$
for all $i,j$. \\

Now,  the  additive  group  $W$  stable under  the  action  of  $\Gamma
=P_{n+1}$ has the property that it contains a non-zero scalar multiple
$\mu  (\e_i)$ for  all  $i$ by  Lemma \ref{complexreflections}.  Hence
$\lambda \mu (R\e_i) \subset W$.  This is [3] of the proposition.
\end{proof}

\begin{notation}  Suppose  $\mathfrak  a$  is  a prime  ideal  in  $R$
invariant under the involution on  $R$ given by $X_i \mapsto X_i^{-1}$
for each $i$.  Consider  the integral domain $A=R/{\mathfrak a}$.  Let
$B$ denote  the invariants  in $A$ under  the involution.  We  get the
corresponding  reduced Gassner representation  $g_n(A)$ on  the module
$\oplus _{i=1}^n  A\e _i$ of $P_{n+1}$.  Write $t_i$ for  the image of
$X_i$  under the  quotient  map $R\ra  A=R/{\mathfrak  a}$.  The  skew
hermitian form $h$ reduced modulo $\mathfrak a$ gives a skew hermitian
form  on $A^n$  and  the image  of  $P_{n+1}$ under  $g_n(A)$ lies  in
$U(h)(A)$. Let $E$ denote the quotient field of $A$. Then $W(A)=\oplus
A \e_i $ is a subgroup of the $E$-vector space $W(E)=\oplus E \e_i$.
\end{notation}

\begin{proposition} \label{gassnerspecialisation} 

{\rm [1]} If $t_1 \cdots t_{n+1}\neq 1$ then $h(A)$ is non-degenerate and 
the representation $g_n(A)$ is irreducible.  \\

{\rm  [2]}  The central  element  $\Delta  _n^2$  of $P_{n+1}$  {\rm [}where
$\Delta  _n =(s_1\cdots  s_n)\cdots (s_1s_2)s_1${\rm ]}  acts by  the scalar
$t_1\cdots t_{n+1}$ on the representation $g_n(A)$. \\

{\rm   [3]}  If   $t_1\cdots   t_{n+1}\neq  1$   then  every   nonzero
$P_{n+1}$-invariant subgroup  of $A^n=\oplus A\e_i$  contains $\lambda
(A^n)$ for some non-zero element $\lambda \in A$.
\end{proposition}

The  proof of Proposition  \ref{irreducible} can  be repeated  for the
quotient $R/{\mathfrak a}$ in place of $R$.

\section{Proof of Theorem \ref{purebraidmainth}}

We will  first prove some  preliminary results on the  reduced Gassner
representation at $d$-th roots of unity. We will then use induction to
deduce Theorem \ref{purebraidmainth} from these results.

\subsection{The reduced Gassner representation at roots of unity}

Consider  the  reduced   Gassner  representation  $g_n(k):  P_{n+1}\ra
GL_n(R)$ where  $R$ s the  Laurent polynomial ring in  $n+1$ variables
$X_i$  with  integral  coefficients  and  $R ^n$  is  the  free  module
$W_n(X)=\oplus  _{i=1}^n R  \e_i$.  We  now specialise  to $X_i\mapsto
t_i=  t^{k_i}$. The  resulting representation  is from  $P_{n+1}$ into
$GL_n(A_d)\subset  GL_n(E_d)$  where $E_d$  is  the $d$-th  cyclotomic
extension of  $\Q$ and  $A_d$ the  ring of integers  in $E_d$,  and is
denoted $g_n(k,d)$. Denote the vector space $W_n(k,d)$.

\begin{lemma}    \label{gassnerirreducible}   The    reduced   Gassner
representation  evaluated at all  primitive $d$-th  roots of  unity is
irreducible if and only if $t_1t_2\cdots t_{n+1}\neq 1$. \\

When $t_1\cdots  t_{n+1}=1$, the  representation space $W_n(k,d)$
of the representation $g_n(k,d)$  contains a non-zero invariant vector
$w$ and the restriction of the quotient representation $W_n(k,d)/E_dw$
to the subgroup  $P_n$ is isomorphic to $g_{n-1}(k,d)$.   If we denote
the   quotient  representation   by   ${\overline  g_n(k,d)}$,   then
${\overline g_n(k,d)}$ is irreducible.\\

If we denote by $\pi _i$ the product $t_1t_2\cdots t_i$, then 
the invariant element $w\in W_n(k,d)$ is given by
\[w=\sum (1-\pi _i)\e_i.\]
\end{lemma}

\begin{proof} If $t_1\cdots t_{n+1}\neq 1$, then the specialisation of
the  Hermitian  form  $h$  at  $t_1,  \cdots,  t_{n+1}$  has  non-zero
determinant,    by   Lemma    \ref{abdul}.    Therefore,    by   Lemma
\ref{complexreflections},    the    representation    $g_n(k,d)$    is
irreducible. \\

Suppose $t_1t_2\cdots t_{n+1}=1$. Consider the element 
\[v=e_1+t_1e_2+\cdots     +    t_1t_2\cdots    t_{n-1}e_n+t_1t_2\cdots
t_{n+1}e_{n+1}.\] This  is the vector part of  the $P_{n+1}$ invariant
element   $x_1x_2\cdots   x_{n+1}$    in   the   semi-direct   product
$K^{ab}\rtimes t^{\Z}$, and is hence  invariant. 
The expression for $v$ in terms of $\e_i$ and $v_{n+1}$ shows that 
\[v=\sum     _{i=1}^n    (1-t_1t_2\cdots     t_i)\e_i    +(1-t_1\cdots
t_{n+1})v_{n+1}.\]  By assumption  on  the $t_i$,  the coefficient  of
$v_{n+1}$  is  zero,  and  hence  $v$  lies  in  the  reduced  Gassner
representation $W_n(k,d)$ (the span of the $\e_i$). \\

Let $P_n$  be the set of pure  braids in the group  generated by $s_2,
s_3, \cdots,  s_n$. The module  $W=W_n(k,d)$ is spanned by  the vector
$v$  and $\e_2,  \cdots, \e_n$.   Therefore, $g_n(k,d)$  restricted to
$P_n$ splits into a direct sum of the modules $Rv$ and $g_{n-1}(k,d)$;
the   latter  is   irreducible  by   part  [1],   since  $t_2t_3\cdots
t_{p+1}=t_1^{-1}\neq   1$.   We   have  therefore   proved   that  the
representation  $\overline  {g_n(k,d)}$  restricted  to $P_n$  is  the
representation  $g_{n-1}(k,d)$  and is  irreducible  for the  subgroup
$P_n$; hence it is irreducible for the bigger group $P_{n+1}$.
\end{proof}

The   lemma  is   due   essentially  to   Abdulrahim  \cite{Abd}   and
\cite{Abd2}.  We have  derived the  lemma since  we need  the explicit
formula for the invariant element, in terms of the basis $\e_i$. \\

\subsection{Gassner representations with 
degenerate Hermitian forms}

Suppose that $t_i$ are all primitive $d$-th roots of unity 
and that  $t_1t_2\cdots t_p=1$. Consider the  basis $\e_1, \cdots,
\e_{p-1}$ of $W$. This contains the element
\[w= \sum _{i=1}^{p-1} (1-t_i)\e_i.\] Therefore, it is part of a basis
of  $W$: $w,\e_2,  \cdots,  \e_{p-1}$. Since  $w$  is invariant,  with
respect to this basis, every element $g$ of the pure braid group $P_p$
has the matrix form
\[\begin{pmatrix} 1  & v  \\ 0_{p-2} &  \alpha \end{pmatrix}  \] where
$\alpha $  is the matrix of  $g$ acting on the  quotient $W/E_dw$ with
respect to the basis $\e_2,\e_3,  \cdots , \e_{p-1}$. The element $C'=
(\Delta ' )^2$ where
\[\Delta '  =(s_2\cdots s_p)(s_2\cdots s_{p-2})\cdots (s_2s_1)(s_2),\]
is  central in  the pure  braid group  $P_{p-2}$ and  hence acts  by a
scalar on the irreducible  representation; the scalar is $t_2t_3\cdots
t_p=t_1^{-1}\neq 1$  (by part [2] of  Proposition \ref{gassnerspecialisation}).  
Therefore, the commutator
\[u=[g,(\Delta ' )^2] = \begin{pmatrix} 1 & v' \\ 0_{p-2} & 1_{p-2}
\end{pmatrix} \] acts by an  upper triangular unitary matrix. 

\begin{proposition} \label{unipotentradical} Let $t_1, \cdots, t_p$ be
primitive $d$-th  roots of unity with $t_1\cdots  t_p=1$. Consider the
representation  $g(p,t):  P_p\ra  U(h)(O_d)$, the  reduced  Gassner
representation $g(p,X)$ specialised to  $X_i \mapsto t_i$. Then\\

[1] The unitary group $U(h)$ has a unipotent radical isomorphic to the
$P_{p-1}$   module   $W^*$,  where   $W$   is   the  reduced   Gassner
representation $g(p-1,t)$. \\

[2] If  $g= s_1^2$ and  $\Delta '$ is  as in the  preceding paragraph,
then the image of the  element $u=[g,(\Delta ')^2 ]$ under the reduced
Gassner representation $g(p,t)$ is not identity. \\

[3] The conjugates $\{huh^{-1}\}$ of $u$ for $h\in P_{p-1}$ generate a
subgroup of finite index  in the integral additive subgroup $W^*(O_d)$
of $W^*$.

\end{proposition}

\begin{proof} Part [1] is obvious. \\

The vector 
\[w=\sum   _{i=1}^  p  (1-t_1\cdots   t_i)\e_i,\]  (notice   that  the
coefficient of $\e_p$ is zero) is orthogonal to all the vectors $\e_1,
\cdots, \e_p$ and is invariant under all of $P_p$. We can consider the
basis  $B'$  given by  $  w,\e_2, \cdots,  \e_p$  of  the vector  space
$W=\oplus   _{i=1}^p  \e_i$.   Now   $s_1^2$  fixes   $w$.  By   Lemma
\ref{epsilonbasis},    we    have    the   equalities    $s_1^2(\e_2)=
\e_2+(1-t_1)\e_1=    \e_2+w-(1-t_1t_2)\e_2-\cdots    -    (1-t_1\cdots
t_p)\e_p$, and  $s_1^2 (\e_i)=1$ for $i\geq 3$.   Therefore the matrix
of $s_1^2$ in this basis $B'$ is of the form
\[g(p,t)(s_1^2)= \begin{pmatrix}  1 & 1 & 0 & \cdots  & 0 \\ 0  & t_1t_2 &
0 & \cdots & 0  \\ \cdots & \cdots & \cdots & \cdots  & \cdots \\ 
0 &  \cdots & 0 & 0 & 1 \end{pmatrix}.\]

Consider the pure braid  group $P_{p-1}'$ generated by $s_2^2, \cdots,
s_p^2$.  The  subspace spanned  $W'$ by $\e_2,  \cdots, \e_p$  is left
stable  under  $P_{p-1}'$  and  the resulting  representation  is  the
Gassner  representation  for  $P_{p-1}'$.  Consequently,  the  element
$(\Delta ') ^2$ acts by the scalar $c= t_2\cdots t_p=t_1^{-1}\neq 1$ on
$W'$.  Moreover, $(\Delta '  )^2 $  fixed the  vector $w$.  Hence with
respect to the basis $B'$ the element $(\Delta ' )^2$ has the matrix
\[g(p,t)((\Delta ')^2)= \begin{pmatrix} 1 & 0 &  0 & \cdots & 0 \\ 0 &
c & 0 & \cdots & 0 \\ \cdots  & \cdots & \cdots & \cdots & \cdots \\ 0
& \cdots & 0  & 0 & c \end{pmatrix}.\] It is  then clear from the above
matrix forms  that the commutator of  $s_1^2$ and $(\Delta  ' )^2$ has
the matrix form
\[g(p,t)(u)=  g(p,t)([s_1^2,  (\Delta  '  )^2])= \begin{pmatrix}  1  &
t_1^{-1}t_2^{-1} (1-c  ^{-1}) & \cdots &  0 \\ 0 &  1 & \cdots  & 0 \\
\cdots  &   \cdots  &  \cdots   &  \cdots  \\   0  &  \cdots  &   0  &
1  \end{pmatrix}.\] In  other  words,  this lies  in  the vector  group
$W^*(O_d)$  and is  a non-identity  element (since  $c\neq  1$).  This
proves part [2]. \\

Consider  the image  (under $g(p,t)$)  of the  group generated  by the
conjugates $huh^{-1}$  of $u$  by elements $h$  (of the  smaller braid
group $P_{p-1}$ generated by  $s_2, s_3, \cdots, s_{p-1}$). This image
may be  identified with an additive  subgroup $A$ of  the vector group
$W'^*$ where $W'^*$ is dual of the reduced Gassner representation $W'$
at $d$-th roots  of unity for $P_{p-1}$.  It follows  from part [3] of
Proposition \ref{gassnerspecialisation} that the additive subgroup $A$
contains a subgroup of finite index in the (dual of the) vector group
\[W'(A_d)\simeq \oplus _{i=2}^ p  A_d\e_i.\] 

\end{proof}

We  now  consider  the  ``next''  pure  braid  group  $P_{p+1}$,  with
$t_1t_2\cdots t_p=1$. The  Gassner representation takes $P_{p+1}$ into
a unitary group $U(h)$.  The product $t_1t_2\cdots t_{p+1}=t_{p+1} \neq 1$ 
since
all the $t_i$ are primitive $d$-th roots of unity. With respect to the
basis  $w,  \e_2, \cdots  \e_{p-1}$  and  $\e_p$,  the matrix  of  the
commutator $u=[s_1^2, (\Delta ')^2]$ now takes the form
\[u=  \begin{pmatrix}  1  &  v  &  \lambda  \\  0_{p-2}  &  1_{p-2}  &
\overline{v} \\ 0  & 0 & 1\end{pmatrix},\] and  the group generated by
the conjugates  $\{mum^{-1}: m\in P_{p-1}\}$  is a subgroup  of finite
index in the  Heisenberg group $H(X)$ where $H(X)$  is the subgroup of
$U(h)$ which acts unipotently on the flag
\[E_d  w  \subset  E_dw+X  \subset  V,\]  where $X$  is  the  span  of
$\e_2\cdots, \e_{p-1}$,  and $V$  is the span  of $\e_1,  \cdots, \e_p$
(the reduced Gassner representation). Hence the Gassner image contains
a  subgroup  of  finite   index  in  the  unipotent  integral 
Heisenberg  group $H(X)(O_d)$.

\begin{proposition}   \label{opposite}  Let  $t_1,t_2,   \cdots,  t_p,
t_{p+1}$ be  primitive $d$-th roots  of unity such  that $t_1t_2\cdots
t_p=1$. Denote by $\Gamma _p$ the image of the Gassner representation
\[g_p(k,d):P_{p+1}\ra  U(h)(O_d)\subset GL_p(A_d).\] Then  

[1]  there exist  two  opposite maximal  parabolic  subgroups $P$  and
$P^{-}$  of $U(h)$  such that  the  image $\Gamma  _p$ intersects  the
integral unipotent radicals $U_P(O_d)$ and $U^{-}_P(O_d)$ in subgroups
of finite index. \\

[2] In  particular, if $K-rank  (U(h))\geq 2$, then the  image $\Gamma
_p$ is an arithmetic subgroup  of {\rm (}i.e. subgroup of finite index
in{\rm )} $U(h)(O_d)$.
\end{proposition}

\begin{proof} Since  $t_1t_2\cdots t_{p+1}=t_{p+1}\neq 1$,  it follows
that  the hermitian form  $h= h_p$  is non-degenerate.  Therefore, the
unitary group $U(h)$ is reductive. Denote by $V=V_p=E_d^p$ the natural
representation of $U(h)$ ($E_d$  is the $d$-th cyclotomic extension of
$\Q$). \\

We have already  seen in the paragraph preceding  the statement of the
proposition,  that $\Gamma  _p$ contains  a subgroup  $U_0$  of finite
index in  the integer  points of the  Heisenberg group  $H(X)$. Recall
that  $H(X)$ is  the unipotent  radical  of a  parabolic subgroup  $P$
(i.e. $P$ is the normaliser of $H(X)$ in the unitary group $U(h)$). We
will now prove  that there exists a conjugate of  $U_0$ in $\Gamma _p$
which is an arithmetic subgroup of an opposite unipotent radical. \\

By assumption, there exists an isotropic
vector $v= \sum _{i=1}^ {p-1} (1-\pi _i)\e_i$ (it is orthogonal to all
the $\e_i$  with $i\leq  p-1$ and is  therefore orthogonal  to itself).
Consequently,  we may find  a basis  
\[v=w_1, \cdots,  w_r, x_1,\cdots, x_s, w_1^*, \cdots,  w_r^*\] of $V$
where  $x_i$  are orthogonal  to  the  $w_i$  and $w_i^*$,  $w_i$  are
isotropic mutually orthogonal vectors, similarly, $w_j^*$ are mutually
orthogonal  isotropic  vectors,  and  $w_j^*(w_i)=\delta  _{ij}$  (the
Kronecker  delta  symbol).  The  intersection  of  the diagonals  with
$U(h)$ then gives a maximal torus defined over $K$, which is maximally
$K$-split over  $K$.  $N(T)$ denotes  the $K$-Weyl group  and $\kappa$
denotes the longest element in the $K$-Weyl group.\\

The following Zariski density statement is very likely true in greater
generality  (cf. Lemma  (11.5) of  \cite{Del-Mos} and  Lemma  (4.4) of
\cite{Looi} ), but we will need only this weaker version in the course
of the proof.

\begin{lemma}   \label{zariskidense}  Suppose  that   $\Gamma  \subset
U(h)(O_d)$  is the image  of the  pure braid  group under  the reduced
Gassner representation.   Suppose that  $\Gamma _n$ contains  a finite
index  subgroup of the  Heisenberg group  which can  be viewed  as the
integral  unipotent  radical  of  a {\rm  (}maximal{\rm  )}  parabolic
$K_d$-subgroup  of $U(h)$ and  that $\Gamma$  acts irreducibly  on the
reduced Gassner representation {\rm (}i.e.  suppose that $t_1t_2\cdots
t_{n+1}\neq  1${\rm  )}.   Then,  the Zariski  closure  of  $\Gamma_n$
contains the special unitary group $SU(h)$.
\end{lemma}

\begin{proof} We  will use Proposition  \ref{Uzariskidense} of section
\ref{alggroups}.   The irreducibility  of  the action  of $\Gamma  _n$
implies  that the  Zariski  closure (intersected  with $SL_n(\C)$)  is
reductive.  Since $\Gamma  _n$ is  assumed to  contain a  finite index
subgroup  of the  group of  integral points  of the  unipotent radical
given by the  Heisenberg group, it follows that if  $H$ is the Zariski
closure of $\Gamma _n$ then $H$ contains $U$, the group of the form in
Proposition \ref{Uzariskidense}.  Then Proposition \ref{Uzariskidense}
implies the Lemma.
\end{proof}

We  continue with the  proof of  Proposition \ref{opposite}.  By Lemma
\ref{zariskidense},  it follows  that $\Gamma  _p$ intersects  the big
Bruhat cell (the  Zariski open set $U\kappa P$ where  $\kappa $ is the
longest  Weyl group element  above). Let  $\gamma $  lie in  the image
$\Gamma _p$  and also  in the big  Bruhat cell.  Then  $\gamma P\gamma
^{-1}=P'$ is opposite to $P$  and hence $\gamma U_0 \gamma ^{-1}$ lies
in  $\Gamma  _p$.  Therefore,   the  first  part  of  the  proposition
follows. \\

To  prove the  second part,  note that  $U(h)(O_d)$ is  a  higher rank
lattice containing $\Gamma$.  By part 1 of the  proposition, $\Gamma $
contains a  finite-index subgroup of integral points  of the unipotent
radical  of  a parabolic  $K$-subgroup.  By Lemma  \ref{zariskidense},
$\Gamma $ is Zariski dense.  Hence by Theorem \ref{bamise}, the second
part follows.
\end{proof}

\subsection{Proof of Theorem  \ref{purebraidmainth}}

Recall the notation: let $k_1,  k_2, \cdots, k_{n+1}$ be integers with
$1\leq k_i  \leq d-1$  and co-prime to  $d$. Let $\omega  _d=e^{2\pi i
/d}$  be  a primitive  $d$-th  root  of  unity; write  $t_i=\omega  _d
^{k_i}$. Let  $\Q(\omega _d)$ be  the $d$-th cyclotomic  extension and
$A_d$  the ring  of integers  in $E_d$;  denote by  $K_d$  the maximal
totally real sub-field  $\Q (2{\rm cos}(2\pi /d))$ of  $E_d$ and $O_d$
the  ring of  integers  in  $K_d$. Denote  by  $g_n(k,d): P_{n+1}  \ra
U(h)(O_d)\subset   GL_n(A_d)$  the   reduced   Gassner  representation
specialised at $X_i\mapsto t_i$. In this section, for ease of notation
(since we use induction and have  to deal with many indices) we denote
by $V_n$ the $E_d$ vector  space spanned by $\e_1, \cdots, \e_n$ (this
is the same as $W_n(k,d)$). \\

Suppose $s\in  B_{n+1}$ is an element  of the {\it  full} braid group,
whose image  in $S_{n+1}$ is  the permutation $\sigma$. Now  $\sigma $
operates  on  the  Laurent  polynomial ring  $\Z[X_1^{\pm  1},  \cdots
X_{n+1}^{\pm1}]$ by permutations of the $X_i$'s.  Given $s$, denote by
$u$ the $n\times n$ matrix $(u_{ij})$ such that
\[\sigma (\e_i)=  \sum _{j=1}^n u_{ji} \e_j. \]  Recall that $B_{n+1}$
operates on the Gassner module  $W\otimes \Omega $ by automorphisms of
the abelian  group , {\it but  not linearly} over  $\Omega$. Let $s\in
B_{n+1}$, $\sigma  $ its image  in $S_{n+1}$.  It easily  follows from
the construction  of the Gassner  representation that, for all  $g$ in
the pure braid group $P_{n+1}$, we have the equality
\[g(n,X)(sgs^{-1})= \sigma (u) \sigma (g(n,X)(g) \sigma (u) ^{-1}.\]
In particular, the image of $g(n,X)$  is a conjugate, by an element in
$GL_n(R)$,  of  the  image   of  the  twisted  representation  $\sigma
(g(n,X))=g(n,\sigma (X))$.   Therefore, if we  wish to prove  that the
specialisation of $g(n,X)$ at some roots of unity is arithmetic, it is
enough  to  prove  it   for  the  ``twisted''  representation  $\sigma
(g(n,X))$  (the  notation  $\sigma  (g(n,X))$ means  that  the  matrix
entries  of $g(n,X)$  which  are elements  of  $R$ acted  upon by  the
permutation $\sigma$ of the variables $X_i$). \\

We now begin the proof of Theorem \ref{purebraidmainth}.

\begin{proof} We first  prove that there exist two  elements $v,v' \in
E_d^n$  which are  linearly independent  and mutually  orthogonal with
respect to the hermitian form. \\

Consider the  $d$ numbers $t_1, t_1t_2,  \cdots, t_1t_2\cdots t_{d-1}$
and $t_1t_2\cdots t_d$.  These are $d$ elements of  the group $\mu _d$
of $d$-th roots  of unity . Therefore, by  the pigeon- hole principle,
we have  two possibilities. (a) One  of these products is  one, or (b)
Two of them coincide. \\

Hence there exists a  subset $I\subset\{ 1,2, \cdots, d-1, d
\}$ consisting of $l$ consecutive  integers such that $\prod _{i\in I}
t_i=1$.  Let $I=\{a+1,  \cdots a+l\}$. Put 
\[v= \sum  _{i=1}^ {l-1} (1-t_{a+1}\cdots t_{a+i}) \e  _{a+i}.\] By an
earlier  computation, $v$  is  orthogonal to  all  the $\e_{a+i}$  with
$1\leq i  \leq a+l$. Note  that in the  expression of $v$ as  a linear
combination of  the $\e_i$, the ``last'' basis  vector $\e_{a+l}$ does
not appear.  \\

Similarly,  there exists  a  subset $J\subset  \{  d+1, \cdots,  2d\}$
consisting  of $m$  consecutive  numbers such  that  $\prod _{j\in  J}
t_j=1$;  let  $J=\{b+1,  \cdots  b+m\}$,  where $b\geq  l+1$. As before, 
\[v' =  \sum _{j=1}^ {m-1}  (1- t_{b+1}\cdots t_{b+j}) \e  _{b+j},\] is
isotropic and is  orthogonal to all the $\e_{b+j}$  with $1\leq j \leq
m$.  Since the  indices of the $\e_{\mu}$ occurring in  $v$ are of the
form  $\mu  =a+i$  with  $i\leq   l-1$,  it  follows  that  $a+i  \leq
d-1$.  Since the  indices of  $\e _{\nu}=  \e_{b+j}$ occurring  in the
expression for  $v'$ are of the  form $\nu =b+j \geq  d+1$, it follows
that $\nu -\mu \geq 2$. Therefore, $v,v'$ are orthogonal. \\

Therefore  the $E_d$-rank  of  the span  of  the vectors
$\e_{a+1}, \cdots , \e_{a+l}$ and $ \e_{b+1}, \cdots, \e_{b+m}$ is at
least two. \\

By  the  remarks preceding  the  beginning  of  the proof  of  Theorem
\ref{purebraidmainth},  we may assume  that $a=0$  and $b=l$,  after a
permutation  $\sigma$ of the  indices, so  that $t_1\cdots  t_l=1$ and
$t_{l+1}\cdots         t_{l+m}=1$.         Therefore        $t_1\cdots
t_{l+m+1}=t_{l+m+1}\neq 1$ and $V_n$ is non-degenerate if $n=l+m$.  By
Proposition \ref{opposite},  the group $\Gamma  _{l+m}$ intersects two
opposite  integral  unipotent  radicals  $U_P^{+}$  and  $U_P^{-}$  in
subgroups  of  finite  index.   By  the conclusion  of  the  preceding
paragraph we  get that for $n=l+m$,  the group $U(h)$  has $K$-rank at
least  two, and  therefore the  group  $\Gamma _n$  is arithmetic,  by
Theorem \ref{bamise}.  \\

If the Theorem is true for  {\it some} $n\geq l+m$, then we will prove
that it is true for $n+1$. There are several cases to consider. \\

Let  $V=V_n$ (resp.  $V_{n+1}$)  be  the span  of  $\e_1, \cdots  ,\e_n$
(resp. $\e_1,  \cdots, \e_{n+1}$).  Let $V_n  '$ be the  span of $\e_2,
\cdots, \e_{n+1}$.  Then, the  intersection $V_n\cap V_n'$ is the span
of  $\e_2,  \cdots,  \e_n$.  Therefore,  $V_n\cap V_n  '$  contains  a
subspace  $W''$  which is  non-degenerate  and  contains an  isotropic
vector (e.g.  take $W''$  to be the  span of $\e_l,  \e_{l+1}, \cdots,
\e_{l+m+1}$). \\

{\it Case 1.} Assume that $V_n,V_n ', V_{n+1}$ are all non-degenerate.
Since $n\geq l+m$  the $K$-rank of $U(h_{n+1})$ and  $U(h_n)$ are both
$\geq 2$. By induction  assumption, $U(V_n)\cap \Gamma $ is arithmetic
and  $U(V_n')\cap  \Gamma$  is  arithmetic. By  Lemma  \ref{inductive},
$U(h_{n+1})$ is also arithmetic. \\

{\it Case 2.} $V_{n+1}$ is non-degenerate but $V_n$ is degenerate. Then
by   Proposition   \ref{opposite},   $U(h_{n+1})\cap   \Gamma   $   is
arithmetic.  Similarly, if  $V_{n+1}$  non-degenerate but  $V_n '$  is
degenerate, $\Gamma _{n+1}$ can be proved to be arithmetic. \\

{\it Case 3.}  $V_{n+1}$ is  degenerate. Then $V_{n+1}$ contains a one
dimensional null space $E_dv$ and by induction, the image of $\Gamma $
in $U(V_{n+1}/Av)$ is arithmetic.  However, by part [3] of Proposition
\ref{unipotentradical},  the group $\Gamma  $ intersects  the integral
unipotent  radical  of  $U(V_{n+1})$   in  a  finite  index  subgroup.
Therefore, $\Gamma $ contains a  finite index subgroup of the integral
unipotent radical of  $U(h)$ and maps onto a  finite index subgroup of
the reductive Levi part of  $U(h)(O_d)$. Hence $\Gamma $ is a subgroup
of finite index. \\

If $n\geq 2d$, then in the  above notation, $n\geq l+m$. Now the three
induction  steps proved  above imply  that  the group  $\Gamma _n$  is
arithmetic. 
\end{proof}

\section{Homology of Cyclic Coverings} \label{cycliccoverings}

In this section, we will view the first $\Q$-homology of certain index
$d$ subgroups of the free  group on $n+1$ generators, with essentially
a direct sum of the reduced Gassner representation evaluated at $d$-th
roots of unity. The proof is a little indirect, since it does not seem
possible to get a natural basis of the homology of index $d$-subgroup,
which  generates  the  relevant  Gassner representation.  Instead,  we
replace  the free  group $F_{n+1}$  on  $n+1$ genetators  with a  free
product of $F_{n+1}$  and an auxilliary $\Z$ (the  free product is then
the free group  on $n+2$ generators).  It seems  easier to identify the
homology  of   a  finite   index  subgroup  of   the  latter,   with  a
specialisation of the Gassner representation. 

\subsection{Image of homology} \label{homologyimage} 
Let  $k_1, k_2, \cdots, k_{n+1}$  be integers co-prime
to $d$ with $1\leq k_i \leq d-1$. Write $k=(k_1,k_2, \cdots k_{n+1})$.
We get  a homomorphism from the  free group $F_{n+1}\ra  \Z/ d\Z$ (the
latter written multiplicatively  as $q^{\Z}/q^{d\Z}$) by sending $x_i$
to  the element  $q^{k_i}$.  Denote  by  $K(k,d)$ the  kernel to  this
map. We have thus an exact sequence
\[1 \ra K_0(k,d)  \ra F_{n+1} \ra \Z/d\Z \ra 1.\]  Being a subgroup of
finite index (in fact of  index $d$) in $F_{n+1}$, the group $K_0(k,d)$
is also free on $n_0'$ generators, with
\[(1-n_0')= d(1-(n+1)),~~{\rm i.e.}~~ n_0'=1+nd.\] Since the $k_i$ are
co-prime to  $d$, the  image of $x_i$  generates $\Z/d\Z$;  hence there
exists an element  $\xi \in F_{n+1}$ such that its  image is $q$. Then
the  elements  $z_i=x_i\xi  ^{-k_i}$  lie in  the  kernel  $K_0(k,d)$;
moreover,  the elements  $\xi  $  and $\{z_i  :~~1\leq  i \leq  n+1\}$
generate $F_{n+1}$.  Hence the $z_i$  and $\xi ^d$ generate the kernel
$K_0(k,d)$  as a normal  subgroup of  $F_{n+1}$: if  we go  modulo the
normal subgroup  generated by the  $z_i$, then the resulting  group is
generated by the image of $\xi$ and maps onto $\Z/d\Z$. \\

The  first homology  group  $V_0= K_0(k,d)^{ab}\otimes  \Q$ with  $\Q$
coefficients of  $K_0(k,d)$ is therefore  a vector space  of dimension
$n_0'=1+nd$ over $\Q$.  We have  already seen that this homology group
is a  module over the quotient  group $\Z/d\Z$, and is  hence a module
over the group ring $\Q[q]/(q^d-1)$.  As a $\Q$ vector space, $V_0$ is
generated  by the  elements  $\xi ^j  (z_i):~(2\leq  i \leq  n+1~~{\rm
and}~~0\leq  j\leq  d-1)$  and  the  element $\xi  ^d$.   Hence  these
elements form a basis of the first homology group over $\Q$.  \\

We first find the invariants $V_0^G$  in $V_0$ under the action of the
group  $G=\Z/d\Z$. Since the  image $q$  of $\xi  $ generates  $G$, it
follows  that  $V_0=(V_0/(1+q+\cdots+q^{d-1}))\oplus  V_0^G$. In  this
decomposition, the element  $q$ may be replaced by  $q^{k_i}$ for any
$i$ since $k_i$ is coprime to $d$. \\

We know that $\xi ^d \in V_0$ is invariant. Moreover, $x_i^d\in V_0$ is
invariant under  conjugation by $x_i$;  but the conjugation  action of
$F_{n+1}$ on  $V_0$ descends to that  of $G$ and  each $x_i$ generates
$G$. Therefore,  $x_i^d$ is invariant  under $G$. We write  $x_i^d$ in
terms of the $x_j=z_i \xi ^{k_i}$:
\[x_i ^d=  z_i(1+q^{k_i}+\cdots +  q^{k_i(d-1)}).\] As an  operator on
$V_0$    multiplication   by    the    element   $M_i=1+q^{k_i}+\cdots
+q^{(d-1)k_i}$    is    zero   on    non-invariants    and   $d$    on
invariants. Therefore, $M_i$ is independent of $k_i$ and hence $x _i ^d =
z_i(1+q+\cdots +q^{d-1})$. In other words the $\Q$-span of the $x_i^d$ is
the space of all invariants $V_0^G$:
 \begin{equation} \label{K_0invariants}  V_0^G= \sum _{i=1}^  {n+1} \Q
[x_i ^d] .  \end{equation}

Consider the  quotient $V_0^{ni} =  V_0/(1+q+\cdots q^{d-1})V_0$ ($ni$
stands for non-invariants).   This is a module over  the quotient ring
$\Q[q]/(1+q+\cdots  q^{d-1})$.   By the discussion  of the  preceding
paragraphs, the first homology  $V_0$ is generated as a $\Q[q]/(q^d-1)$
module by  (the images  in the abelianisation  of $K_0(k,d)$ of  ) the
elements $z_i$ and  by $\xi ^d$.  Since $\xi $  commutes with $\xi ^d$
it follows that  in the homology group, $q(\xi  ^d)=\xi ^d$; hence the
augmentation  ideal  of  the  group ring  $\Q[q]/(q^d-1)$  kills  $\xi
^d$. Therefore we  see that the $z_i$ form a basis  of the free module
$V_0^{ni}$ as a module over the ring $R_d= \Q[q]/(1+q+ \cdots q^{d-1})$:
\[V_0^{ni}=R_d^n.\]

We   also  have   the   free  group   on   $n+2$  generators   written
$F_{n+1}*t^{\Z}$, with a natural  inclusion of $F_{n+1}$ in $F_{n+2}$.
On $F_{n+2}$ we have a  homomorphism into $\Z/d\Z$ by sending $x_i$ to
$q^{k_i}$ and  $t$ to the  standard generator $q$. Denote  by $K(k,d)$
the kernel to this map; we have a short exact sequence
\[1  \ra K(k,d)  \ra  F_{n+1}*t^{\Z}\ra \Z/d\Z  \ra  1.\] For  similar
reasons, $K(k,d)$  is free on $n'$  generators where $n'$  is given by
the  formula   $n'=1+(n+1)d$.  Thus  the  first   homology  group  $V=
H_1(k,d)^{ab}\otimes \Q$  has dimension $(n+1)d+1$ over $\Q$  and is a
module  over the  group ring  $\Q[q]/(q^d-1)$.   If we  go modulo  the
action  of  the  linear  transformation  $(1+q+\cdots  +q^{d-1})$  the
resulting  quotient  of the  homology  group  of  $K(k,d)$ is  denoted
$V^{ni}$. Then, as before,
\[V^{ni}= R _d ^{n+1},\] is a free module over the ring $R_d$. \\

The   natural   inclusion   of   $F_{n+1}$   in   the   free   product
$F_{n+1}*t^{\Z}$ induces  a map $\phi$  of the homology groups  of the
kernels  $K_0(k,d)$ and $K(k,d)$,  which is  also equivariant  for the
action  of the  group ring  $\Z[\Z/d\Z]$. Thus  we get  maps  of $R_d$
modules $\phi ^{ni} : V_0^{ni}\simeq R_d^n\ra V^{ni}\simeq R_d^{n+1}$.
Since,  as $\Q$-vector  spaces, the  image of  $\phi :V_0\ra  V  $ has
codimension $\geq  $ $(n+1)d-nd=d$, it  follows that (at the  level of
``non-invariants'')   the  image   of  $V_0^{ni}$   in   $V^{ni}$  has
codimension $\geq d-1$. \\

\subsection{Image of the Gassner Module}

The pure braid group $P_{n+1}$  acts on $F_{n+1}$, and acts trivially
on  the abelian  quotient $\Z/d\Z$.   We  take the  trivial action  of
$P_{n+1}$   on  $t^{\Z}$   and   get  an   action   of  $P_{n+1}$   on
$F_{n+1}*t^{\Z}$.   Hence  $P_{n+1}$   acts  on   $K_0(k,d)$   and  on
$K(k,d)$. The map $K_0(k,d) \ra K(k,d)$ is equivariant for this action
as well, and the action of $P_{n+1}$ on the modules $K_0(k,d)^{ab}$and
$K(k,d)^{ab}$  is equivariant  for  the action  of  the product  group
$\Z/d\Z\times P_{n+1}$.\\

We had the exact sequence 
\[1\ra  K \ra  F_{n+1}*  F_{n+1}^{ab} \ra  F_{n+1}^{ab}  \ra 1,\]  and
realised   $K^{ab}$  as   a  module   over  the   group   algebra  $R=
\Z[F_{n+1}^{ab}]$.   By    replacing   $R$   by    the   larger   ring
$R'=R[\frac{1}{(1-X_1)\cdots (1-X_{n+1})}]$,  we found the  direct sum
decomposition
\[K^{ab}\otimes_R  R'\simeq W\otimes  R' \oplus  R'v_{n+1},\]  of $R'$
modules ({\it  but not of $P_{n+1}$ modules}).  In this decomposition,
$W$  was  the   reduced  Gassner  representation  $g_n(X):  P_{n+1}\ra
GL_n(R)$.   We   now  get  a  map   from  $F_{n+1}*F_{n+1}^{ab}$  onto
$F_{n+1}*t^{\Z}$ given by $x_i\mapsto x_i\in F_{n+1}$ and $X_i \mapsto
t^{k_i}$.  This induces a  map from  the abelianised  kernels $K^{ab}$
into $K(k,d)^{ab}$,  with the image of $K^{ab}$  being precisely the
specialisation  $X_i \mapsto t^{k_i}$;  in other  words, the  image of
$K^{ab}$ is the Gassner  representation evaluated at $(t^{k_1}, \cdots
,t^{k_{n+1}})$.  It follows  that if  $v$ is  the invariant  vector in
$K^{ab}$ then
\[v=(1-X_1)\e_1+  \cdots   +  (1-X_1X_2\cdots  X_n)\e_n+  (1-X_1\cdots
X_{n+1})v_{n+1}. \] Then its image in $K(k,d)^{ab}$ is also invariant;
moreover, if $t_1\cdots  t_{n+1}=1$ then the image of  $v$ lies in the
image of  $W$, namely it  is the image  of the element  $\sum _{i=1}^n
(1-t_1\cdots t_i)\e_i$.  \\

The polynomials $1-q^{k_i}$ and $1+q+\cdots q^{d-1}=(1-q^d)/(1-q)$ are
coprime  since $k_i$ and  $d$ are  coprime. Therefore,  $1-q^{k_i}$ is
invertible   in  the   ring   $R_d=\Q[q]/(1+q+\cdots  q^{d-1})$,   and
therefore,   the   map  $R\mapsto   \Q[q]/(q^d-1)$   induces  a   ring
homomorphism  $R'\ra  R_d$  (recall  that $R=\Z[X_1^{\pm  1},  \cdots,
X_{n+1}^{\pm 1}]$ and that $R'$  is obtained from $R$ by inverting all
the elements  $1-X_i$). Hence $V^{ni}= H_1(K(k,d)^{ab},\Q)/(1+q+\cdots
q^{d-1})$ is naturally a module over $R'$.  Therefore, $V^{ni}$ is the
(full)   Gassner  representation   evaluated  at   $(q^{k_1},  \cdots,
q^{k_{n+1}})$.  The  module $V^{ni}$ contains image  of the sub-module
$W\otimes R'$, under the map $R' \mapsto R_d$. However, $W$ is spanned
by the ``commutator'' elements
\[\e_i=\frac{1}{(1-X_i)(1-X_{i+1})}[x_i,x_{i+1}].\] Denote by $\e_i '$
the image of $\e_i$ in $V_0$.  Therefore, $\e_i '$ lie in the image of
the  abelianised kernel $K_0(k,d)^{ab}\otimes  _R R'$  (the commutator
$[x_i,x_{i+1}]$  certainly lies in  the image  of $K_0(k,d)$;  but the
inverted elements  of $R'$  map into elements  which have  inverses in
$R_d$  and  hence $\e_i'$  lies  in  the  image tensored  with  $R'$).
Consequently,  the   image  of  $V_0^{ni}=   K_0(k,d)\otimes  \Q$  has
codimension $\leq d-1$:
\[ V^{ni}/ (image (W\otimes R') =Rv/ (1+q+\cdots+q^{d-1}), \]
and the latter space has dimension $\leq d-1$. \\

The  conclusion of  the preceding  subsection  \ref{homologyimage} now
implies that the module  $V_0^{ni}$ maps injectively into $V^{ni}$ and
that the  image is precisely the  image of $W\otimes  R'$, the reduced
Gassner representation evaluated  at $(q^{k_1}, \cdots, q^{k_{n+1}})$.
We have proved the following

\begin{theorem}   \label{K_0homology}   The  representation   of
$P_{n+1}$ on the homology group
\[V_0^{ni}=(K_0(d,k)^{ab}\otimes    \Q)/(1+q+\cdots   q^{d-1}),\]   is
isomorphic  to  the  reduced  Gassner  representation  specialised  at
$X_i\mapsto q^{k_i}$, and as a module over $R_d$ it is the free module
$R_d^n$  {\rm (}where  $R_d$ is  the quotient  ring $\Q[q]/(1+q+\cdots
+q^{d-1})${\rm )}. We have therefore, the decomposition
\[V_0=\bigoplus _{e\mid d, ~~e\geq 2}g_n(k,e),\] of the (non-invariant
part of  the ) homology  of $K_0(k,d)$ as  a sum of the  {\it reduced}
Gassner representations $g_n(k,e)$.
\end{theorem}

Denote by $Q_0(k,d)$ the quotient  of the free group $K_0(k,d)$ by the
smallest  subgroup  $N$ normalised  by  $F_{n+1}$  and containing  the
``unipotent''  elements $x_1^d,  x_2  ^d, \cdots,  x_{n+1}^d$ and  the
element $(x_1x_2\cdots  x_{n+1})^{d/r}$, where  $r$ is the  g.c.d.  of
the sum $k_1+k_2+\cdots+k_{n+1}$ and the number $d$ (If we view $F$ as
the fundamental group of the  punctures Riemann surface, then $F$ is a
subgroup of  $SL_2(\R)$ and the  loops $x_i$ around the  punctures are
unipotent elements  in $SL_2(\R)$; this  is the reason we  have called
the $x_i$ unipotent  elements. We do not use the fact  that $F$ may be
viewed as  a subgroup of  $SL_2(\R)$). The quotient map  $K_0(k,d) \ra
Q_0(k,d)$ induces a corresponding map on the $\Q$-homology:
\[\phi: V_0=  H_1(K_0(k,d),\Q) \ra X_0=H_1(Q_0(k,d)),\Q).\]  Since the
kernel  $N$ to  the  quotient  map is,  by  assumption, normalised  by
$F_{n+1}$, it  follows that the foregoing  map $\phi $  on homology is
equivariant for the action of $\Z/d\Z$; therefore, $\phi $ is a map of
$\Q[q]/(q^d-1)$ modules. Correspondingly, we get a map
\[\phi:     V_0^{ni}\ra     X_0^{ni}=     H_1(Q_0(k,d),\Q)/(1+q+\cdots
+q^{d-1}).\] The vector  space $X_0$ does not have  any invariants for
the group $\Z/d\Z$,  because invariants in the quotients  $X_0$ are in
the  image of  invariants in  $V_0$ ($G$  is a  finite group  and the
modules    are   $\Q$-vector    spaces).    Secondly,   by    equation
\ref{K_0invariants},  the only  invariants in  $V_0= H_1(K_0(k,d),\Q)$
are the span of the ``unipotent'' classes $[x_1^d], \cdots, [x_{n+1}^d]$
which   lie    in   the   kernel   of    the   map   $K_0(k,d)^{ab}\ra
Q_0(k,d)^{ab}$. Therefore, $X_0^{ni}=X_0$. \\

The action of the group $P_{n+1}$  on the free group $F_{n+1}$ is such
that each generator $x_i$ of $F_{n+1}$ goes into a conjugate of itself
(see subsection  (\ref{artinfree})).  Moreover, the product element
$x_1\cdots x_{n+1}$  is invariant under  all of $B_{n+1}$.   Hence the
normal subgroup $N$ is stable under the action of the pure braid group
$P_{n+1}$, and therefore, the  homology group $H_1(Q_0(k,d), \Q)$ is a
$P_{n+1}$ module and  the map $\phi$ is equivariant  for the action of
$P_{n+1}$  as  well.  We  have  the  following  corollary  of  Theorem
\ref{K_0homology}.

\begin{corollary} \label{Q_0homology}  The representation of $P_{n+1}$
on the quotient $X_0$ is a direct sum
\[X_0\simeq \bigoplus  _{e \mid  d,~~e\geq 2} {\overline  g_n(k,d)},\] of
the   quotients  ${\overline   g_n(k,d)}$  of   the   reduced  Gassner
representations by the {\rm (} possibly one dimensional{\rm )} 
space of invariants
\end{corollary}

\begin{proof} Since $X_0=X_0^{ni}$ it  follows that the elements $x_i^d
\in  V_0$ map  to zero  in $X_0$.  Hence the  quotient $X_0$  is $V_0$
modulo   the   $\Z[q]$   -   module  generated   by   $g=(x_1x_2\cdots
x_{n+1})^{d/r}$: the other elements $[x_i^d]$ in $V_0$ map to zero. \\

We now  deal with the element  $(x_1x_2\cdots x_{n+1})^{d/r}$.  Recall
that $r$  is the g.c.d.  of the integers  $d$ and $\sum  _{i=1} ^{n+1}
k_i$.   Then, the  element $g_{\infty}=  (x_1x_2\cdots x_{n+1})^{d/r}$
viewed   as   an   element  of   $F_{n+1}/K_{n+1}(d)^{(1)}$   (written
multiplicatively)  lies in  $K_{n+1}(d)^{ab}$.  Put $t_i=q^{k_i}$  and
$\pi  =t_1t_2\cdots  t_{n+1}$.   In  $K_{n+1}(k,d)^{ab}$  the  element
$g_{\infty}$ of the abelian group (written additively) is of the form
\[w=(x_1x_2\cdots     x_{n+1})^{d/r}=      v(1+\pi     +\cdots     \pi
^{d/r-1})=\lambda v.\]  In this formula, $v$ is  the invariant element
in $K^{ab}$ encountered before:
\[v=(1-t_1)\e_1+\cdots + (1-t_1t_2\cdots  t_{n+1})\e_n.\] We need only
check that  this element $g_{\infty}= \lambda  v$ goes to  zero in the
$e$-th component  of the decomposition  $H_1(K_0(k,d),\Q)\simeq \oplus
g_n(k,e)$, exactly when $g_n(k,e)$ has an invariant vector. \\

If  $g_n(k,e)$  has  an  invariant  vector,  then  the  element  $\pi=
t_1t_2\cdots  t_{n+1}$ of  $\Z[\Z/d\Z]$ maps  to $1$  in  the quotient
$\Z/e\Z$ of the group $G=\Z/d\Z$.   Hence the projection to the $e$-th
factor of $w$  is a non-zero scalar multiple of  the projection of $v$
since $\lambda =d/r \neq 0$ at the $e$-th place. \\

If $g(k,e)$ does not have an invariant vector, then $\pi =t_1t_2\cdots
t_{n+1}  $   is  not  1  in  $\Z/e\Z$;   hence  $\lambda  =\frac{1-\pi
^e}{1-\pi}= 0$  and our  element $g_{\infty}$ goes  to zero.   We have
therefore verified that corresponding to the decomposition
\[V_0=H_1(K_0(k,d),\Q)/(1+q+\cdots  +q^{d-1})\simeq  \bigoplus _{e\mid
d~e\geq  2} g_n(k,e)\]  of  representations of  the  pure braid  group
$P_{n+1}$, the quotient  $X_0=V_0^{ni}/\Q[G]g_{\infty}$ has a corresponding
decomposition
\[X_0=H_1(Q_0(k,d),\Q) \simeq \bigoplus _{e\mid d} \overline{g_n(k,e)}.\] 
as representations of the pure braid group $P_{n+1}$. 

\end{proof}

\begin{corollary} \label{homologyarithmetic}  If $n\geq 2d$,  then the
image  of the representation  of $P_{n+1}$  on the  space $W_0$  is an
arithmetic group.
\end{corollary}

\begin{proof} For $e \geq 2$ dividing $d$, denote by $G_e$ the unitary
group  of  the  skew  hermitian  form ${\overline  h}$  on  the  space
${\overline g_n(k,e)}$  .  We have  seen from the  preceding corollary
that if  $\Gamma $ is  the image of  the action of $P_{n+1}$  on $X_0$
then $\Gamma$ is contained in the product $\prod _{e\mid d} G_e(O_e)$.
Suppose first  that $e\geq  3$. Then by  Theorem \ref{purebraidmainth}
(since $n\geq  2d\geq 2e$), the image  of $\Gamma $  in $G_e(O_e)$ has
finite    index.    If    $e=2$   then    the   method    of   Theorem
\ref{purebraidmainth} does not apply. However, for $e=2$, is is a well
known  theorem  of   \cite{A'C}  that  the  image  of   $\Gamma  $  in
$G_e(O_e)=Sp  _{2e}(\Z)$  has   finite  index.   Therefore,  by  Lemma
\ref{products}, $\Gamma $  has finite index in the  product $\prod _{e
\mid d} G_e(O_e)$.

\end{proof}

\section{Connection with Monodromy} \label{monodromysection}

\subsection{Some Cyclic Coverings of ${\mathbb P}^1$} 
\label{topologybraid}

Let $a_1,a_2, \cdots,a_{n+1}$ be distinct complex numbers; write $S_a$
for     the     complement     in     $\C$    of     these     points:
$S_a=\C\setminus\{a_1,a_2,\cdots,a_{n+1}\}$. The  fundamental group of
$S_a$, once  a base point is  chosen, may be identified  with the free
group on $F_{n+1}$  generated by small circles $x_i$  going around the
point $a_i$  counterclockwise once (and  joined to the  preferred base
point by an arc which avoids all the points $a_j$ and has zero winding
number around all the points $a_j$). The map $S_a\ra \C^*$ defined by
\[x\mapsto
(x-a_1)^{k_1}(x-a_2)^{k_2}\cdots(x-a_{n+1})^{k_{n+1}}=P_a(x),\]
induces a homomorphism $F_{n+1}\ra  q^{\Z}$, which sends each $x_i$ to
$q^{k_i}$.   Here, $q$  is a  small circle  around zero  in $\C^*$
which runs counterclockwise exactly once.\\

For  future reference,  note that  the loop  around infinity  lying in
$S_a$ represents  (the inverse  of) the product  element $x_1x_2\cdots
x_{n+1}$ and  that this element is  invariant under the  action of the
braid group $B_{n+1}$ on the free group $F_{n+1}$. \\

The affine  variety $\C^*={\mathbb G}_m$  admits a cyclic  covering of
order $d$  given by $z\mapsto  z^d$ from ${\mathbb G}_m$  to ${\mathbb
G}_m$.   The  covering  may  be  realised  as  the  space  $\{(x,y)\in
\C^*\times \C^*:y^d=x\}$ and the covering map is the first projection.
Pulling this covering back to $S_a$ we get a cyclic covering of $S_a$,
realised as the space
\[X_{a,k}=\{(x,y)\in              S_a\times     \C^* :
~~y^d=(x-a_1)^{k_1}(x-a_2)^{k_2}\cdots(x-a_{n+1})^{k_{n+1}}\},\]   with
the  first   projection  being  the  covering  map   from  $X_a$  onto
$S_a$. Therefore, under the identification of the fundamental group of
$S_a$ with $F_{n+1}$, the fundamental group of $X_{a,k}$ is identified
with $K_0(d,k)$. \\

As  the collection $a$  varies, we  get a  collection $\mathcal  P$ of
monic  polynomials $P_a$  of degree  $n+1$ which  have  distinct roots
$a_i$ occurring with given multiplicities $k_i$, and if ${\mathcal Q}$
denotes the variety
\[(w,x,P)\in \C^*\times  \C \times {\mathcal  P}: w= P(x),\]  then the
projection on to the third coordinate gives a fibration over $\mathcal
P$ with fibre at $P$ being  $S_a$ (here $a$ is the collection of roots
of $P$). We therefore get  a monodromy action of the fundamental group
of ${\mathcal P}$ on the  fundamental group $F_{n+1}$ of the fibre. We
have the following basic theorem of E. Artin.

\begin{theorem}  {\rm  (}  Artin   {\rm)}  The  fundamental  group  of
$\mathcal P$ is a subgroup of 
the Braid Group $B_{n+1}$ and contains the pure braid group $P_{n+1}$. 
The monodromy action of
$P_{n+1}$  on  $\pi  _1(S_a)\simeq  F_{n+1}$  is the  usual  action  of
$P_{n+1}$ on $F_{n+1}$ defined in section (\ref{artinfree}).
\end{theorem}

Consequently, the monodromy action on the fibre of the fibration
\[\{(y,x,a)\in    \C^*\times   \C\times   {\mathcal    C}:   y^d=\prod
(x-a_i)^{k_i} \] over ${\mathcal C}$  is the usual action of $P_{n+1}$
on  the   subgroup  $K_{n+1}(d)\simeq  \pi   _1(X_{a,k})$;  therefore,
$P_{n+1}$ acts on the first homology $K_{n+1}(k,d)^{ab}$ of $X_{a,k}$,
and  this gives  the monodromy  action. Therefore,  the first  part of
Proposition   \ref{monodromyandgassner} follows   from  (the   Hurewicz
Theorem and) Theorem \ref{K_0homology}.

\subsection{The  Compactification  of  $X_{a,k}$} 

Denote by  $X_{a,k}^*$ the  compactification of the  affine $X_{a,k}$;
$X_{a,k}$ is  a compact Riemann surface with  finitely many punctures;
hence $X_{a,k}^*$ is  a smooth projective curve obtained  by filing in
these punctures.

Now the  covering map $X_{a,k} \ra  S_a$ is such  that these punctures
lie  over the  points $a_i$  or  else over  the point  at infinity  of
$S_a$. If a  puncture lies over some $a_i$, then the  image of a small
loop  around the  puncture in  $F_{n+1}$  is $x_i^d$  (since $k_i$  is
coprime to $d$);  if the puncture lies above  infinity, then the image
of  a small  loop  around the  puncture  in $F_{n+1}$  is  a power  of
$x_1x_2\cdots x_{n+1}$  the loop  around infinity; therefore,  such an
element is invariant under the action of the braid group.\\

The  mapping of $\pi  _1(X_{a,k})\ra \pi  _1(X_{a,k}^*)$ is  such that
these  loops around  the punctures  generate  the kernel,  by the  van
Kampen theorem; consequently,  the elements $x_1^d, \cdots, x_{n+1}^d$
map  to zero  in  $H_1(X_{a,k}^*,\Z)$ and  the element  $(x_1x_2\cdots
x_{n+1})^{d/r}$  maps to  $0$.   Therefore, the  fundamental group  of
$X_{a,k}^*$  may  be  identified   with  the  quotient  $Q_0(k,d)$  of
Corollary   \ref{Q_0homology}.    Therefore,   the  second   part   of
Proposition     \ref{monodromyandgassner} follows     from    Corollary
\ref{Q_0homology}. \\

The arithmeticity of the monodromy (Theorem \ref{cyclicmonodromy}) now
follows from Corollary  \ref{homologyarithmetic} since the homology of
$X_{a,k}^*$ is the homology of $Q_0(k,d)$. \\

\noindent {\bf  Acknowledgements:} I am  very grateful to  Madhav Nori
for the crucial remark that  an earlier proof for the arithmeticity of
the image of  the Burau representation would also  go through for the
Gassner representation). \\

The  support of the  JC Bose  fellowship for  the period  2008-2013 is
gratefully acknowledged.


\begin{thebibliography}{JPSH}


\bibitem{Abd}  M.  Abdulrahim,  Complex  specialisations  of  the
reduced   Gassner   representation    of   the   pure   braid   group,
Proc. Amer. Math.Soc, {\bf 125} (1997), no. 6, 1617-1624.\\

\bibitem{Abd2}  M. Abdurahim, A  faithfulness criterion  for the
Gassner representation of the pure  braid group, Proc. Amer. Math. Soc
{\bf 125}, (1997), 1249-1257.  \\

\bibitem{A'C}  N.    A'Campo,  Tresses,  monodromie   et  groupes
symplectique, Comment. Math. Helvetici, {\bf 54} (1987), 318-327. \\

\bibitem{Ba-Mi-Se}  H.Bass,   J.Milnor  and  J-P.   Serre,
Solution of  the Congruence subgroup  problem for $SL_n $  ($n\geq 3$)
and $Sp_{2n}$ ($n\geq 2$), Publ. IHES, {\bf 33}, (1967), 59-137. \\

\bibitem{Beu-Hec} F.Beukers and G. Heckman, Monodromy for the
hypergeometric  function $nF_{n-1}$, Invent.   Math. {\bf  95} (1989),
325-354. \\

\bibitem{Bir} J.Birman,  Braids, links and  mapping class groups,
{\bf 82}, Annals of Math. Studies, Princeton University Press, 1974.\\

\bibitem{Bor-Ti}   A.Borel  and  J.Tits,   Groupes  reductifs,
Publ. IHES.{\bf 27}, (1965), 55-150.\\

\bibitem{Del-Mos}  P.Deligne and  G.D.  Mostow,  Monodromy of
hypergeometric   functions   and   non-lattice   integral   monodromy,
Publ.Math. IHES {\bf 63}, (1986), 5-89. \\

\bibitem{Gr-Sch} P.Griffiths and W.Schmid, Recent deveopements
in  Hodge theory:  a discussion  of techniques  and  results, Discrete
subgroups of Lie groups  and applications to moduli (Internat. Colloq,
Bombay 1973), pp 31-127 Oxford University Press, Bombay 1975. \\

\bibitem{Gru-Lub}  F.Grunewald and A.Lubotzky,  Linear representations
of the automorphism group of a  free group, GAFA {\bf 18} (2009), no.5
1564-1608. \\

\bibitem{KM}  M.Kapovich   and  J.Millson,  Quantization   of  bending
deformations of  polygons in  $E^3$, hypergeometric integrals  and the
Gassner representation, Canad.Math.Bull. {\bf 44} (2001), {\bf no. 4},
36-60. \\

\bibitem{Long}  D.D.Long,  On  linear representations  of  braid
groups, Transactions of the A.M.S. Vol 311, No.2, (1989), 535-560. \\

\bibitem{Looi}   E.  Looijenga,  Uniformization   by  Lauricella
functions,  an  overview of  the  theory  of Deligne-Mostow,  207-244,
Progr.Math.{\bf 260} Birkhauser, Basel (2007). \\

\bibitem{Looi2}  E.  Looijenga, Prym  Representations  of the  Mapping
Class Groups, Geom. Dedicata {\bf 64} (1997), no.1, 60-83. \\

\bibitem{Mar} G.A.Margulis, Discrete subgroups of semi-simple Lie
groups, Ergebnisse der Mathematik und ihrer Grenzgebiete 3.Folge. Band
{\bf 17}, Springer-Verlag (1991).\\

\bibitem{Mc}   C.  Mcmullen,  Braid   groups  and   Hodge  theory,
Math.Annalen, Online First, 27th March, (2012).
http:// dx.doi.org/10.1007/s00208-012-0004-2 .  \\

\bibitem{Mos}  G.D. Mostow,  Generalised Picard  lattices arising
from  half-integral  conditions, Publ.Math.   IHES,  {\bf 63}  (1986),
91-106. \\

\bibitem{Nor}   M.V.Nori,   A   nonarithmetic  monodromy   group,
C.R.Acad.Aci, Paris, Ser I, math.302 (1986), no. 2, 71-72.\\

\bibitem{Ra}   M.S.   Raghunathan,  A   note  on   generators  for
arithmetic   subgroups  of  algebraic   groups,  Pacific   Journal  of
Mathematics, {\bf 152} (1991), 365-373. \\


\bibitem{Sar} P.  Sarnak, Notes on  thin groups, MSRI  hot topics
workshop, Feb 2012). \\


\bibitem{Sq}  C.Squier,  The   Burau  representation  is  unitary,
Proc. Amer. Math. Soc. {\bf 90} (1984), 199-202. \\


\bibitem{Ti}  J.   Tits,  Syst'eme`s  generateurs  de  groupes  de
congruence, C.R.Acad. Sci. Paris, Serie A {\bf 283} (1976), 693. \\


\bibitem{Va} L. Vaserstein,  The structure of classical arithmetic
groups  of rank  greater than  1 (English  Translation),  Math.  USSR,
Sbornik {\bf 20} (1973), 465-492. \\
 

\bibitem{Ve}  T.N.Venkataramana,  On  systems  of  generators  for
arithmetic subgroups  of higher rank groups, Pacific  Journal of Math.
{\bf 166} no.1 (1994), 193-212. \\

\bibitem{Ve2}  T.N.Venkataramana,   Arithmeticity  of  the  Burau
representation at roots of unity, (to appear in Annals of Mathematics).

\end{thebibliography}
\end{document}